\documentclass[11pt]{article}
\usepackage[vmargin=2.5cm]{geometry}
\usepackage{graphicx} % Required for inserting images
\usepackage{amssymb}
\usepackage{amsthm}
\usepackage{amsmath}
\usepackage{booktabs}
\usepackage{longtable}
\usepackage{mathbbol}
\usepackage{setspace}
\usepackage{parskip}
\usepackage{filecontents}
\newtheorem{theorem}{Theorem}[section]
\newtheorem{lemma}{Lemma}[section]
\newtheorem{remark}{Remark}[section]
\usepackage[numbers]{natbib}
\allowdisplaybreaks
\title{The Waring-Goldbach problem in short intervals under the Generalized Riemann Hypothesis}
\author{Yaojie Guo}
\date{November 2026}

\begin{document}
\maketitle
\begin{abstract}
Let $N$ be a sufficiently large odd integer, and let $k$ be an odd prime.
We study the representation of $mN^k$ as a sum of $m$ $k$-th
powers of primes in $[N-N^\theta,N+N^\theta]$.
Assuming the Generalized Riemann Hypothesis (GRH), we prove that for every
$\theta>1/2$ there is a solution for all sufficiently large $N$, provided
$k$ is a sufficiently large odd prime and
\[
m>\frac{6k+10\theta-6}{2\theta-1}.
\]
Moreover, the number of solutions satisfies
\[
I(N)\sim\mathfrak G(Q)(\log N)^{-m}
N^{\theta(m-1)+1-k}.
\]
For fixed $\theta$, this requires only $m=O_\theta(k)$ prime
variables, the first bound of linear order in $k$; the previous best was
$k(k+1)$. Thus, under GRH, the short-interval Waring--Goldbach problem is
solvable with a number of variables linear in $k$, the order predicted by
the Hardy--Littlewood circle method for Waring's problem.
\end{abstract}

\setcounter{tocdepth}{2}
\tableofcontents
\newpage

\section{Introduction}

\subsection{The Waring and Waring--Goldbach problems}
Let $k$ be an odd prime and let $N$ be a sufficiently large odd integer.
The classical Waring problem asks for the least integer $m=G(k)$ such that
every sufficiently large integer is a sum of $m$ $k$-th powers:
\[
n=x_1^k+\cdots+x_m^k,\qquad x_i\in\mathbf{Z}.
\]
Hua proved that $G(k)\le 2^k+1$ \cite{Hua1963}. In 1992, Wooley proved
\[
G(k)\le k\log k+k\log\log k+O(k)
\]
(see \cite{wooley1992largeimprovements}), and the proof of the Vinogradov
mean value conjecture by Bourgain, Demeter and Guth
\cite{bourgain2016proofmainconjecturevinogradovs} led to further improvements.

We now combine the Waring problem with the Goldbach problem. Let
$s=v_p(k)$ and define
\[
\gamma=
\begin{cases}
s+2,&p=2,\ s>0,\\
s+1,&\text{otherwise},
\end{cases}
\qquad
R_k=\prod_{(p-1)\mid k}p^\gamma.
\]
For odd prime $k$, we have $R_k=2$; since $N$ and $k$ are odd,
$N^k\equiv1\pmod2$, and the congruence condition
$mN^k\equiv m\bmod{R_k}$ is automatic for every $m$. The Waring--Goldbach problem asks for representations
\[
mN^k=p_1^k+\cdots+p_m^k,\qquad p_i\in[N-N^\theta,N+N^\theta],
\]
with $p_i$ prime.

\subsection{Previous results}
Wei and Wooley \cite{wei2014sumspowersequalprimes} proved that for
$s>\max\{6,2k(k-1)\}$ one may take
\[
\theta_{k,s}\le
\begin{cases}
19/24,&k=2,\\
4/5,&k=3,\\
5/6,&k\ge4.
\end{cases}
\]
Their method relies on two exponential-sum estimates: one for the major arcs
and one for the minor arcs. The first is
\[
\sum_{N<n<N+N^\theta}\Lambda(n)e(n^k\alpha)
\ll N^{\theta+\epsilon-\frac{\theta-1/2}{K}},
\]
where $K=2(k^2-k)(k^2-k+2)$, and the second is the Vinogradov mean value
theorem
\[
\int_{[0,1]^k}\Bigl|\sum_{n\le X}e(f(n))\Bigr|^{2m}\mathrm d\alpha
\ll_\epsilon X^\epsilon\bigl(X^{2m-k(k+1)/2}+X^m\bigr),
\]
proved in \cite{bourgain2016proofmainconjecturevinogradovs}. Huang
\cite{huang2016exponentialsumsprimesshort} obtained $\theta=3/4$, and
Matom\"aki and Shao
\cite{matomaki2019discorrelationprimesshortintervals} improved this to
$\theta=2/3$. Further estimates appear in
\cite{wang2022waringgoldbachproblemshortintervals}.

The error terms in theorems on primes in arithmetic progressions, such as
the Bombieri--Vinogradov theorem (or the
Barban--Davenport--Halberstam theorem; see \cite{BDH} for a recent
quantitative form), are closely related to estimates for exponential sums. For instance, the Bombieri--Vinogradov
theorem is philosophically related to
\[
\sum_{q\le Q}\max_{(a,q)=1}
\Bigl|\psi(x;q,a)-\frac{x}{\varphi(q)}\Bigr|
\rightsquigarrow
\sum_{q\le Q}\sum_{\chi\ne\chi_0}|\psi(x,\chi)|
\rightsquigarrow
\sum_{q\le Q}\sum_{\chi\ne\chi_0}
\Bigl|\sum_{n\le x}\chi(n)\Lambda(n)\Bigr|,
\]
where $Q=x^{1/2}(\log x)^{-B}$; the relations are philosophical rather than
exact. The last sum is again an average of exponential sums, and it is the
main tool for the major arcs.

Our approach differs from that of
\cite{wei2014sumspowersequalprimes}. There, the condition $m>k(k-1)$ is a natural
barrier in the major-arc analysis (see formula (6.1) of that paper). By choosing
$k$ to be an odd prime, we can overcome this barrier. The remaining difficulty is
\[
\sum_{q\le P}\sum_{\substack{1\le a\le q\$a,q)=1}}
\varphi^{-s}(q)C(q,a)^{s-j}
e\Bigl(-\frac{an}{q}\Bigr)I_j(q,a),
\]
and the properties of $I_j(q,a)$ are controlled by
\[
J_\nu(g)=\sum_{1\le r\le P}[g,r]^{1+\nu-s/2}
\sum_{\chi\bmod r,\ \chi\ne\chi_0}
\|W(\chi,\lambda)\|_{L^\infty(-1/(rQ),1/(rQ))},
\]
\[
K_\nu(g)=\sum_{1\le r\le P}[g,r]^{1+\nu-s/2}
\sum_{\chi\bmod r,\ \chi\ne\chi_0}
\|W(\chi,\lambda)\|_{L^2(-1/(rQ),1/(rQ))}.
\]
In this paper we shall give a strong estimate, under GRH and QGRH($c$), for a quantity that is actually equivalent to $W(\chi,\lambda)$.
\subsection{The main theorem}
We now state our main result. Assume $\theta>1/2$ and
\[
m>\frac{6k+10\theta-6}{2\theta-1}.
\]
\begin{theorem}
Under GRH, for sufficiently large $N$, the equation
\[
mN^k=p_1^k+\cdots+p_m^k
\]
has a solution, provided $\epsilon=o(1/k)$; the number $I(N)$ of
solutions satisfies
\[
I(N)\sim\mathfrak G(Q)(\log N)^{-m}N^{\theta(m-1)+1-k}
\]
\end{theorem}
Set, with the convention $C_1(1,k)=1$,
\[
C_q(a,k)=\sum_{\substack{1\le l<q\$l,q)=1}}
e\Bigl(\frac{al^k}{q}\Bigr).
\]
The singular series is
\[
\mathfrak G(Q)=
\sum_{q<Q}\sum_{\substack{1\le a\le q\$a,q)=1}}
C_q(a,k)^{m}
e\Bigl(-\frac{a mN^k}{q}\Bigr)
\varphi(q)^{-m}.
\]
\begin{theorem}
Let $k$ be an odd prime. With the notation of Section 2, under GRH we have the following estimate
\[
I_{11}\asymp
\mathfrak G(Q)
\frac{N^{\theta(m-1)+1-k}}{(\log N)^{m}}.
\]
\end{theorem}

\begin{theorem}
With the notation of Section 2, let $I_{21}$ denote the contribution of the minor arcs. Under GRH, for any $\epsilon>0$, we have
\[
I_{21}\ll
N^{({\theta}/{2}+{1}/{4})m+\frac{k+3\theta-1}{2}+\epsilon}(\log N)^{m}
\]
\end{theorem}

The previous best result for all $n$ requires $s>k(k+1)$ prime variables
\cite{wang2022waringgoldbachproblemshortintervals} (the bound
$s>k(k+1)/2$ there is for almost all $n$); under GRH we improve this to
$m=O(k)$ prime variables. This is the first bound of linear order in
$k$. The setting is substantially harder than the classical Waring
problem: the prime variables lie in intervals of length $N^\theta$ with
$\theta$ arbitrarily close to $1/2$, so the primality constraint and the
short-interval constraint must be handled simultaneously. In the classical
problem the analogous threshold is the Hardy--Littlewood conjecture
$G(k)\le k+1$, while the best unconditional bound is
$G(k)\ll k\log k$; the present result reaches the conjectured order in
this harder setting, under GRH and with a constant depending on $\theta$.
To prove the theorems, we convert the exponential sum over primes into an
ordinary exponential sum, thereby reducing the problem to the classical
Waring problem in short intervals.\\

For the minor arcs, we approximate the denominators by $Q$-smooth
square-free numbers, expand the indicator of the approximation condition
via a Fourier series, and apply the large sieve. This is the content of
Sections 3 and 5.
\section{Preliminaries}
\subsection{Notation and exponential sums}
Before turning to the main problem, we fix notation and record several theorems that will be required in the ensuing computations.

Throughout, $m>2$ and $k>2$ are integers, and $N$ is an odd positive integer, with $N$ assumed to be substantially larger than both $m$ and $k$. We write ``$\gg^*$'' and ``$\ll^*$'' for ``much larger than'' and ``much smaller than'', respectively, and ``$f\gg g$'' and ``$f\ll g$'' for ``$g=O(f)$'' and ``$f=O(g)$''.
We always assume that for any function $f$, $f(N)\in\mathbf{Z}$ once $f(N)$ is sufficiently large. The symbols $x_i,y_i$ always denote ordinary integers. For the sake of brevity, in this paper we do not distinguish between different arbitrarily small positive numbers $\epsilon$, even sometimes ignore it in the notation.
We adopt the notation
\[
e(x)=e^{2\pi i x}
\]
We write the exponential sums as
\[
S(\alpha)=\sum_{N-N^{\theta}<u<N+N^{\theta}}e(\alpha u^k),\qquad
T_j(\alpha)=\sum_{N-P_j<u<N+P_j}e(\alpha u^k),
\]
together with their prime analogues
\[
S_p(\alpha)=\sum_{N-N^{\theta}<p<N+N^{\theta}}e(\alpha p^k),\qquad
T_{p,j}(\alpha)=\sum_{N-P_j<p<N+P_j}e(\alpha p^k).
\]
In addition, let
\[
S_p^*(\alpha)=\sum_{N-N^{\theta}<n<N+N^{\theta}}e(\alpha n^k)\Lambda(n).
\]
\subsection{The systems and the arc decomposition}
These definitions are essentially those of \cite{wei2014sumspowersequalprimes}. Let $I$ denote the number of solutions. Fix $\tau>2Q^2$, $(h,q)=1$, and $h\leq Q$. The basic arcs and the major arc are
\[
I(h,q)=
\left\{
    a\in(0,1):\left|\frac{h}{q}-a\right|<\frac{1}{\tau}
\right\},\qquad
\mathfrak{M}_1=\bigcup_{q<Q}\bigcup_{h<q,(h,q)=1}I(h,q),
\]
while the minor arc is
\[
\mathfrak{M}_2=[0,1]-\mathfrak{M}_1
\subseteq\bigcup_{Q<q<\tau}
\left\{
    a\in(0,1):\left|\frac{h}{q}-a\right|<\frac{1}{\tau_0\tau^{1/2-\epsilon}}
\right\}
\] 
 where $\tau_0\asymp\tau/\log N$ and $\epsilon>0$ is a given small positive number. Here $q$ is required to be a $Q$-smooth number, meaning that in the factorization $q=p_1^{\alpha_1}\cdots p_n^{\alpha_n}$ one has $p_i\leq Q$ for all $1\leq i\leq n$. This assertion will be proved in Section 3. In fact, we will prove that one may take $\alpha_i=1$ and $|\mu(q)|=1$, with the estimate that for a $p|d$, $d>pQ$.
\[
Q=N^{\theta-1/2},\tau=N^{k+\theta-1}\log\log N
\]
The above sets overlap the major arc only on a set of measure
arbitrarily close to zero. The integral over the major arc is
\[
I_{11}=\int_{\mathfrak{M}_1}S_p^{m}(t)e\left(-tmN^k\right)\mathrm{d}t,
\]
and the integral over the minor arc is
\[
I_{21}=\int_{\mathfrak{M}_2}S_p^{m}(t)e\left(-tmN^k\right)\mathrm{d}t.
\]
Thus the total number of solutions is $I=I_{11}+I_{21}$. The major-arc
integral is estimated in Section 4, and the minor-arc integral in Section 5.
The treatment of $I_{21}$ is different. In Section 5 we use Lemma 5.1 to transform $I_{20}$ into the sum $A$. It is natural to expect that $A$ can be bounded by the quantities obtained
from GRH, but the natural approximation cannot overcome the difficulty caused
by the large modulus; this suggests that we should consider smooth numbers. The $Q$-smooth approximation resolves the difficulties caused by the CRT by separating the moduli one at a time.
\subsection{Auxiliary estimates}
We now record several theorems that will be used below. For some we give references in place of proofs; for the others we supply short proofs.
\begin{lemma}
For $\alpha\in\mathfrak{M_2}$, let
\[
S(\alpha)=\sum\limits_{N_1<n<N_2}e(\alpha n^k).
\]
Let $K=2(k^2-k)(k^2-k+2)$; then for sufficiently large $N$ we have
\[
S_p(\alpha)\ll N^{\theta+\epsilon-\frac{\theta-1/2}{K}}
\]
\end{lemma}
The conclusion can be found in \cite{wei2014sumspowersequalprimes}.

\begin{lemma}[Hua's Inequality]
    Let $P(x)=hx^k\in\mathbf{Z}[x]$, and define the rational exponential sum
    \[
    C_q(h,k)=\sum_{1\leq l\leq q,(l,q)=1}e\left(\frac{P(l)}{q}\right).
    \]
    Then
    \[
    |C_q(h,k)|\leq C(k)q^{1/2}.
    \]
\end{lemma}

For the sum over reduced residues, Weil's bound gives the exponent $1/2$, as stated above. For the complete sum, the classical Hua bound gives the exponent $1-1/k$; see \cite{Hua1963} and \cite{PP1986}.
\begin{lemma}
Let $A>0$ be any given constant. Then for any $Q>1$ and $x>3$, if
\[
E(x;q,a)=\pi(x;q,a)-\frac{\mathrm{Li}\,x}{\varphi(q)},
\]
then under GRH,
\[
    \sum_{q\leq Q}\sum_{\substack{1\leq a\leq q\$a,q)=1}}\max_{x<u<2x}|E(u;q,a)|^k\ll Qx^{k/2}(\log x)^2.
\]
\end{lemma}

\begin{proof}
 We use the standard identity
\[
\pi(x;q,a)-\frac{\mathrm{Li}\,x}{\varphi(q)}=\psi(x;q,a)-\frac{x}{\varphi(q)}+O\left(x^{1/2}(\log x)^2\right).
\]
Hence
\begin{align*}
\sum_{q\leq Q}\sum_{\substack{1\leq a\leq q\$a,q)=1}}\max_{x<u<2x}|E(u;q,a)|^k
&=\sum_{q\leq Q}\sum_{\substack{1\leq a\leq q\$a,q)=1}}\varphi^{-k}(q)\sum_{\chi\neq\chi_0}|\bar{\chi}(a)|^k|\psi(x,\chi)|^k\\
&<\sum_{q\leq Q}\sum_{\substack{1\leq a\leq q\$a,q)=1}}\varphi^{-k}(q)\sum_{\chi\neq\chi_0}|\psi(x,\chi)|^k.
\end{align*}
Under GRH, when $\theta>1/2$,
\[
\psi(x,\chi)\ll x^{1/2}(\log x)^2.
\]
Therefore
\[
\sum_{q\leq Q}\sum_{\substack{1\leq a\leq q\$a,q)=1}}\max_{x<u<2x}|E(u;q,a)|^k
\ll x^{k/2}\sum_{q\leq Q}\sum_{\substack{1\leq a\leq q\$a,q)=1}}\varphi^{-k}(q)(\varphi(q)-1)
\ll Qx^{k/2}(\log x)^2.
\]
This proves the lemma.
\end{proof}
\section{Q-smooth square-free numbers}
\subsection{The approximation set and the counting theorem}
In this section we consider $Q$-smooth square-free numbers and prove the following theorem.
Set
\[
S(Q,T)=\left\{T/\log T<y\leq T:P_+(y)\leq Q, \mu^2(y)=1, \exists p|n:n>pQ\right\}
\]
We then consider the set of Diophantine approximations
\[
M(\alpha;T^\beta)=\{(h,q):(h,q)=1,h<q,q\in S(Q,T^\beta),|\alpha-h/q|<\frac{1}{qT}\}
\]
For almost every $\alpha$ satisfying the following condition, the theorem holds on the minor arcs:
\[
\alpha=h_0/q_0+a/(q_0\tau),a\in(0,1),(h_0,q_0)=1,Q<q_0<\tau=T^{2+2\epsilon},
\]
where $\epsilon$ is a small number. When $|M(\alpha,T^{2+2\epsilon})|>1$, the covering property holds; here we choose $T^{2+2\epsilon}=\tau$. The following theorem holds uniformly in $T$; thus $T$ here is
merely a symbol satisfying the condition in Theorem 3.1. We take $k$ to be a
sufficiently large integer.
\begin{theorem} The following estimate holds when $\log T/\log Q\asymp k$ and $\beta>2$:
\[
|M(\alpha;T^{\beta})|\sim\rho(u)\frac{12}{\pi^{2}}T^{\beta-1},
\qquad u=\frac{\beta\log T}{\log Q},
\]
where $\rho$ is the Dickman function and $\alpha$ satisfies
\[
\alpha=h_0/q_0+a/(q_0T),a\in(0,1),(h_0,q_0)=1,Q<q_0<T
\]
\end{theorem}
\subsection{The density of $Q$-smooth square-free numbers}
First we need to prove a lemma. Let $x\ge y\ge2$ and $u=\log x/\log y\ge1$. Note
\[
R(x,y)=\sum_{\substack{n\le x,P(n)\le y\\ \exists p|n:n>pQ}}{\mu^{2}(n)}
\]
and
\[
R^*(x,y)=\sum_{\substack{n\le x,P(n)\le y}}{\mu^{2}(n)}
\]
\begin{lemma}
The following estimate holds true:
\[
R^*(x,y)=x\rho(u)\frac{6}{\pi^{2}}(1+O(2^{-\frac{u}{3}}))
\]
\end{lemma}
\begin{proof}
We have
\begin{align*}
R^*(x,y)=\sum_{\substack{n\le x\\P(n)\le y}}{\mu^{2}(n)}
=\sum_{\substack{d\le\sqrt x\\P(d)\le y}}
\sum_{\substack{m\le x/d^{2}}}\mu(d)1_{P(m)\leq y}
\end{align*}
Note $\Psi(x,y)=|\{n\leq x:P_+(n)\leq y\}|$. From \cite{G.Tenenbaum} we have
\[
R^*(x,y)= \sum_{\substack{d\le\sqrt x\\P(d)\le y}}\mu(d)\Psi(\frac{x}{d^2},y)
\]
Then $\Psi(x,y)=x\rho(u)$ and $u=\log x/\log y$. $\rho(u)$ is a Dickman function satisfying $u\rho'(u)=-\rho(u-1)$. So
\begin{align*}
\sum_{\substack{d\le\sqrt x\\P(d)\le y}}\mu(d)\Psi(\frac{x}{d^2},y)=&\sum_{\substack{d\le\sqrt x\\P(d)\le y}}\mu(d)\frac{x}{d^2}\rho(u-\frac{2\log d}{\log y})\\
\ll&\sum_{\substack{d\le\sqrt x\\P(d)\le y}}\mu(d)\frac{x}{d^2}\rho(u)+x\max_{u/2\le u_0\le u}|\rho'(u_0)|\sum_{\substack{d\le\sqrt x\\P(d)\le y}}\frac{\mu(d)}{d^2}\\
=&\sum_{\substack{d\le\sqrt x\\P(d)\le y}}\mu(d)\frac{x}{d^2}\rho(u)(1+O(2^{-\frac{u}{3}}))
\end{align*}
This estimate comes from the derivative of the Dickman function,
\[
|\rho'(u_0)|\ll u^{-u}\leq 2^{-u/3}
\]
Hence
\[
R^*(x,y)=x\rho(u)\sum_{\substack{d\le\sqrt x\\P(d)\le y}}\frac{\mu(d)}{d^2}+O(2^{-u/3}\sum_{\substack{d\le\sqrt x\\P(d)\le y}}\frac{\mu(d)}{d^2})
\]
By the following estimate for the Euler product,
\[
\sum_{\substack{d\le x\\P(d)\le y}}\frac{\mu(d)}{d^{2}}
=\prod_{p\le y}\left(1-\frac1{p^{2}}\right)+O\!\left(\frac1x\right)
  =\frac{6}{\pi^{2}}+O\!\left(\frac1y+\frac1x\right)
\]
Thus
\[
R^*(x,y)=x\rho(u)\frac{6}{\pi^{2}}\bigl(1+O(2^{-u/3})\bigr).
\]
This completes the proof.
\end{proof}
Consequently,
\[
R^*(T_0,Q)=o(R^*(T,Q)).
\]
Now we come to consider $R^*(x,y)$. By the contradict statement gives $n<pQ<Q^2$, we know
\[
R^*(x,y)-R(x,y)<R(Q^2,y)=Q^2\rho(\frac{2\log Q}{\log y})\frac{6}{\pi^{2}}(1+O(2^{-\frac{u}{3}}))
\]
which means
\[
R(x,y)\sim R^*(x,y)
\]
\subsection{The Fourier expansion and the large sieve}
We now prove Theorem 3.1. The gap between two consecutive $Q$-smooth square-free numbers is controlled by $R(x,Q)$:
\[
R(x,Q)=\frac{6x}{\pi^{2}}\rho\Bigl(\frac{\log x}{(\theta-1/2)\log N}\Bigr)(1+O(2^{-\frac{u}{3}})).
\]
We detect the condition on $\{q\alpha\}$ by a Fourier expansion. Let ${1'}_A$ denote the smooth truncation of the indicator function, and set $\epsilon=1/T$. Then
\[
{1}_{\|t\|\leq\epsilon}'(\alpha)=(2\epsilon+\sum\limits_{k\not=0}\frac{\sin 2\pi k\epsilon}{k\pi}e(qk\alpha))\widehat{\psi}_\epsilon(t)
\]
Indeed, this follows from the identity
\[
\widehat{1'}_{{\|q\alpha\|\leq}\epsilon}=\widehat{1}_{\|t\|\leq\epsilon}(t)\times\widehat{\psi}_\epsilon(t),
\]
where, for a suitable function $f$, $\widehat f$ denotes its Fourier transform, and the smoothing function is defined by
\[
{\psi}_\epsilon(t)=
\begin{cases}
C\epsilon^{-2}\exp(-\frac{1}{1-|t/\epsilon^2|^2}),&\,|t|<\epsilon^2,\\
0,&\,|t|\geq \epsilon,
\end{cases}
\]
where $C$ is the normalization constant. We next estimate the error introduced by the smoothing. We write the convolution as
\[
{1}_{\|t\|\leq\epsilon}(t)*{\Phi}_\epsilon(t)=\int_{-\infty}^{+\infty}1_{\|t\|\leq\epsilon}(x-y){\psi}_\epsilon(y)\mathrm{d}y.
\]
Then
\[
{1}_{\|t\|\leq\epsilon}(t)*{\psi}_\epsilon(t)-{1}_{\|t\|\leq\epsilon}(t)\leq\int_{-\infty}^{+\infty}|1_{\|t\|\leq\epsilon}(x-y)-1_{\|t\|\leq\epsilon}(x)|{\psi}_\epsilon(y)\mathrm{d}y.
\]
Since ${\psi}_\epsilon(t)$ is supported in $(-\epsilon,\epsilon)$, the contribution from this interval is the main term, and the error is
\begin{align*}
\int_{-\epsilon^2}^{\epsilon^2}|1_{\|t\|\leq\epsilon}(x-y)-1_{\|t\|\leq\epsilon}(x)|{\psi}_\epsilon(y)\mathrm{d}y=O(\epsilon^2)
\end{align*}
Then the error in the Fourier series is
\[
\int_{0}^{2\pi}\sin{nt}({1}_{\|t\|\leq\epsilon}(t)*{\psi}_\epsilon(t)-{1}_{\|t\|\leq\epsilon}(t))\mathrm{d}t\ll \frac{1}{n}\epsilon^2
\]
Thus the error in the indicator is negligible, and we split the sum into two parts:
\begin{align*}
|M(\alpha;T)|=&\sum\limits_{q\in S(Q,T^\beta)}\sum\limits_{0<|k|\leq T}\frac{\sin 2\pi k\epsilon}{k\pi}e(qk\alpha)\widehat{\psi}_\epsilon(k)\\
+&\sum\limits_{q\in S(Q,T^\beta)}\sum\limits_{|k|>T}\frac{\sin 2\pi k\epsilon}{k\pi}e(qk\alpha)\widehat{\psi}_\epsilon(k)+2\epsilon R(T^\beta,Q)+O(\epsilon^2\log T)
\end{align*}
Denote the exponential sum over $Q$-smooth square-free numbers by
\[
F_k(\alpha)=\sum\limits_{q\in S(Q,T^\beta)}e(qk\alpha).
\]
To calculate the sum $F_k(\alpha)$, a natural idea is to control the sum over $Q$-smooth numbers by the density of the smooth numbers. We start by separating the conditions of square-freeness and $Q$-smoothness as
\begin{align*}
F_k(\alpha)=&\sum_{\substack{d\le T^{\beta/2}\\P(d)\le Q}}
\sum_{\substack{m\le T^{\beta}/d^{2}}}\mu(d)e(md^2k\alpha)1_{P(m)\leq Q}
\end{align*}
Then the first sum of $M(\alpha,T)$ can be estimated as
\begin{align*}
&\sum\limits_{0<|k|\leq T}\frac{\sin 2\pi k\epsilon}{k\pi}\widehat{\psi}_\epsilon(k)\sum\limits_{q\in S(Q,T^\beta)}e(qk\alpha)\\
\leq&\Big(\sum_{\substack{d\le T^{\beta/2}\\P(d)\le Q}}\sum_{\substack{m\le T^\beta/d^{2}}}\Big|\sum\limits_{0<|k|\leq T}\widehat{\psi}_\epsilon(k)\frac{\sin 2\pi k\epsilon}{k\pi}e(md^2k\alpha)\Big|^2\Big)^{1/2}\\
\times&
\Big(\sum_{\substack{d\le T^{\beta/2}\\P(d)\le Q}}\sum_{\substack{m\le T^\beta/d^{2}}}\mu^2(d)1_{P(m)\leq Q}\Big)^{1/2}
\end{align*}
Thus the large sieve inequality applies. The second factor is bounded by
\begin{align*}
\sum_{\substack{d\le T^{\beta/2}\\P(d)\le Q}}\sum_{\substack{m\le T^\beta/d^{2}}}\Big|\sum\limits_{0<|k|\leq T}\widehat{\psi}_\epsilon(k)\frac{\sin 2\pi k\epsilon}{k\pi}e(md^2k\alpha)\Big|^2\leq(T+\delta^{-1}-1)\sum\limits_{0<|k|\leq T}4\epsilon^2\exp(-k^2\epsilon^2)
\end{align*}
where $\delta$ is determined by the separation property of $d^2m\alpha$, and is defined as $\delta=\min_{r}\|d^2m\alpha\|$. By the form of $\alpha=h_0/q_0+a/(q_0\tau)$ with $a\in(0,1)$, we have
\[
d^2ma\ll T^\beta
\]
Thus
\[
d^2m\alpha=\frac{d^2mh_0}{q_0}+\frac{a'}{q_0\tau}
\]
Then
\[
d^2m\alpha-d_1^2m'\alpha=\frac{(d^2m-d_1^2m')h_0}{q_0}+\frac{a(d^2m-d_1^2m')}{q_0\tau}.
\]
The first term vanishes only when
\[
\exists b\in\mathbf{Z} :d^2m-d_1^2m'=bq_0.
\]
Thus
\[
d^2m\alpha-d_1^2m'\alpha=ab/\tau.
\]
Thus $\delta^{-1}\ll T$. In the other case, since $\tau=T^{\beta+\epsilon}$, we have
\[
d^2m\leq T^\beta=o(\tau)
\]
Thus $\delta^{-1}\le q_0\le\tau$. Together with Lemma 3.1, this gives
\[
\ll(T+q_0-1)\sum\limits_{0<|k|\leq T}4\epsilon^2\exp(-k^2\epsilon^2)=O(1)
\]
and from Lemma 3.1,
\[
\sum_{\substack{d\le T^{\beta/2}\\P(d)\le Q}}\sum_{\substack{m\le T^\beta/d^{2}}}\mu^2(d)1_{P(m)\leq Q}=R(T^\beta,Q)
\]
Returning to the Fourier series, we obtain
\[
\Big|\sum\limits_{0<|k|\leq T}\frac{\sin 2\pi k\epsilon}{k\pi}\widehat{\psi}_\epsilon(k)F_k(\alpha)\Big|\ll R^{1/2}(T^\beta,Q)\ll T^{\beta/2}
\]
For the second sum, the factor $\widehat{\psi}_\epsilon(k)$ decays faster than any power of $k$, so
\[
\sum\limits_{|k|>T}\frac{1}{k\pi}\widehat{\psi}_\epsilon(k)|F_k(\alpha)|=O(1)
\]
We thus obtain
\[
|M(\alpha;T^{\beta})|\sim \rho(u)\frac{12}{\pi^2}T^{\beta-1},
\qquad u=\frac{\beta\log T}{\log Q}.
\]
Taking $\beta=2+\epsilon$ completes the proof of Theorem 3.1 and yields
$|M(\alpha,T^{2+\epsilon})|\sim \rho(u)\frac{12}{\pi^2}T^{1+\epsilon}$.
\begin{remark}
Let $\tau=T^{2+2\epsilon}$. Then $|M(Q,\tau)|\sim\tau^{1/2+\epsilon}$ and $T=\tau^{1/2-\epsilon}$. From the theory of uniform distribution, the set of $Q$-smooth
square-free numbers is likely uniformly distributed in $[0,1]$; this means
\[
S^*(Q,T)=\left\{yT^{-1}:y\leq T:P_+(y)\leq Q, \mu^2(y)=1\right\}
\]
is uniformly distributed in $[0,1]$. By the theory of star discrepancy for random distributions, the exponent $1/2$ is the essential resolution of a random array, and it cannot be improved without using the arithmetic structure of $Q$-smooth square-free numbers.
\end{remark}
\section{The Major-Arc Integral}
\subsection{The major arcs and the residue-class reduction}
We now turn to the transformation of the integral over the major arcs.
The procedure is essentially the same as in
\cite{wei2014sumspowersequalprimes}. There are, however, two differences
that make it worthwhile to write out the details here. First, under the
Generalized Riemann Hypothesis (GRH) we have a better estimate for the
exponential sum than the one used in
\cite{wei2014sumspowersequalprimes}. Second, the integral considered here
has a different form from the corresponding integral in
\cite{wei2014sumspowersequalprimes}. We therefore give a complete account,
with the intermediate steps explained.

For readers who are not specialists in the circle method, it may help to
keep the following picture in mind. For a real number $t$, put
\[
S_p(t)=\sum_{N_1<p<N_2}e(tp^k).
\]
The integral
\[
\int_{\mathfrak M_1}S_p^{m}(t)e(-tN_0)\,\mathrm dt
\]
is a Fourier coefficient: it selects from the $m$-th power of the
exponential sum exactly those tuples of primes $(p_1,\dots,p_{m})$ for
which
\[
p_1^k+\cdots+p_{m}^k=N_0.
\]
The set $\mathfrak M_1$ consists of the major arcs, that is, small
neighbourhoods of rational points $a/q$ with small denominator $q$. On
such an arc the exponential sum has a large, well-understood main term;
away from the major arcs, the sum oscillates and is treated separately.
The purpose of this section is to extract the main term on the major arcs
and to show that it has the form ``singular series times singular
integral''.
In this section, we set
\[
N_0=mN^k,\qquad N_1=(\frac{N_0}{m})^{1/k}-N^{\theta},\qquad N_2=(\frac{N_0}{m})^{1/k}+N^{\theta}.
\]
Thus $N_1$ and $N_2$ are the endpoints of a short interval centred at
$(N_0/m)^{1/k}$, of length $2N^\theta$. In the intended application the
centre is $N$, so $N_0$ is the value produced by $m$ copies of $N^k$.

The definition of $N_0$ in this section is the same as above, and the same
notation $N_0$ will be used in Section 5.

We begin by separating the residue class of $p$ modulo $q$ from the size
of $p$. Since $t=a/q+z$, we have
\begin{align*}
S_p(z+a/q)
=&\sum\limits_{N_1<p<N_2}e((z+a/q)p^k)\\
=&\sum\limits_{0\leq l\leq q-1}
  \sum\limits_{N_1<p<N_2,\ p\equiv l\bmod q}
  e((z+a/q)p^k)\\
=&\sum\limits_{0\leq l\leq q-1}e(al^k/q)
  \sum\limits_{N_1<p<N_2,\ p\equiv l\bmod q}e(z p^k).
\end{align*}
We first prove the following lemma.
\begin{lemma}
For any real number $z$, any sufficiently large $q\le N^{\theta-1/2}=Q$, and any $(h_0,q)=1$, we have
\begin{align*}
\sum_{q<Q}\sum_{h<q,(h,q)=1}\int_{|t|<1/\tau}|A_1^{m}-A_2^{m}|\mathrm{d}z\ll_\epsilon N^{\theta(m-1)-k+1-\theta+\frac12(\theta-\frac12)+\epsilon}
\end{align*}
where 
\begin{align*}
A_1=&C_q(h,k)\sum\limits_{N_1<p<N_2,p\equiv h\bmod q}e(tp^k)\\
A_2=&C_q(h,k)\sum\limits_{N_1<p<N_2,p\equiv h_0\bmod q}e(tp^k)
\end{align*}
Here $h_0$ is any positive integer with $h_0<q$ and $(h_0,q)=1$.
\end{lemma}
\begin{proof}
For any $l_0,l_1$, we first show that
\begin{equation}
\sum\limits_{N_1<p<N_2,p\equiv l_0\bmod q}e(z p^k)=\sum\limits_{N_1<p<N_2,p\equiv l_1\bmod q}e(z p^k)+O\Big(\sum_{N_1<u<N_2}\max_{(l,q)=1}|E(u,N_2;q,l)|\Big)
\end{equation}
holds for any $z\in(0,1)$. Taking into account the smoothing of $\pi(x;q,a)$, we obtain
\begin{align*}
\sum\limits_{N_1<p<N_2,p\equiv l_i\bmod q}e(zp^k)\sim&\int_{N_1}^{N_2}e(zu^k)\mathrm{d}\pi(u;q,l_i)\\
\sim&\int_{N_1}^{N_2}e(zu^k)\mathrm{d}\frac{\mathrm{Li}u}{\varphi(q)}+\int_{N_1}^{N_2}e(zu^k)\mathrm{d}E(u;q,l_i)\\
=&\int_{N_1}^{N_2}\frac{e(zu^k)}{\varphi(q)\log{u}}\mathrm{d}u+\int_{N_1}^{N_2}e(zu^k)\mathrm{d}E(u;q,l_i).
\end{align*}
for $i=0,1$. Hence
\begin{align*}
&\sum\limits_{N_1<p<N_2,p\equiv l_0\bmod q}e(z p^k)-\sum\limits_{N_1<p<N_2,p\equiv l_1\bmod q}e(z p^k)\\
=&\sum_{N_1<u<N_2}e(zu^k)(E(u;q,l_0)-E(u;q,l_1)-E(u-1;q,l_0)+E(u-1;q,l_1))\\
=&e(zN_2^k)(E(N_2;q,l_0)-E(N_2;q,l_1))-e(zN_1^k)(E(N_1;q,l_0)-E(N_1;q,l_1))\\[4pt]
-&\sum_{N_1<u<N_2}(E(u;q,l_0)-E(u;q,l_1)-(E(N_2;q,l_0)-E(N_2;q,l_1)))(e(zu^k)-e(z(u-1)^k))
\end{align*}
Define
\[
E(x,y;q,l)=E(x;q,l)-E(y;q,l)
\]
By Taylor expansion,
\[
\bigl|e(zu^k)-e(z(u-1)^k)\bigr|\ll\frac{N^{k-1}}{\tau}
\]
Hence the estimate above gives
\begin{align*}
&\bigg|\sum\limits_{N_1<p<N_2,p\equiv l_0\bmod q}e(z p^k)-\sum\limits_{N_1<p<N_2,p\equiv l_1\bmod q}e(z p^k)\bigg|\\
\ll&\frac{N^{k-1}}{\tau}\sum_{N_1<u<N_2}\max_{(l,q)=1}|E(u,N_2;q,l)|+\max_{N_1<u_0<N_2}\max_{(l,q)=1}|E(u_0,N_2;q,l)|
\end{align*}
Then
\begin{align*}
&\bigg|\Big|\sum\limits_{N_1<p<N_2,p\equiv l_0\bmod q}e(z p^k)\Big|^{m}-\Big|\sum\limits_{N_1<p<N_2,p\equiv l_1\bmod q}e(z p^k)\Big|^{m}\bigg|\\
\ll&\frac{N^{k-1}}{\tau}\sum_{N_1<u<N_2}\max_{(l,q)=1}|E(u,N_2;q,l)|\times\Big|\sum\limits_{N_1<p<N_2,p\equiv l_0\bmod q}e(z p^k)\Big|^{m-1}\\
+&\max_{N_1<u_0<N_2}\max_{(l,q)=1}|E(u_0,N_2;q,l)|\times\Big|\sum\limits_{N_1<p<N_2,p\equiv l_0\bmod q}e(z p^k)\Big|^{m-1}
\end{align*}
Substituting this into the statement of Lemma 4.1 and using the trivial bound for the prime exponential sum, we obtain
\begin{align*}
&\sum_{q<Q}\sum_{l<q,(l,q)=1}\int_{|t-h/q|<1/\tau}|A_1^{m}-A_2^{m}|\mathrm{d}z\ll S_1+S_2
\end{align*}
where
\begin{align*}
S_1=&N^{\frac{1}{2}-\theta}\sum_{q<Q}\sum_{h<q,(h,q)=1}\sum_{N_1<u<N_2}C_q(h,k)\max_{(l,q)=1}|E(u,N_2;q,l)|\\
\times&\int_{|t|<1/\tau}\Big|\sum\limits_{N_1<p<N_2,p\equiv l_0\bmod q}e(z p^k)\Big|^{m-1}\mathrm{d}z\\
S_2=&\sum_{q<Q}\sum_{h<q,(h,q)=1}C_q(h,k)\max_{N_1<u_0<N_2}\max_{(l,q)=1}|E(u_0,N_2;q,l)|\\
\times&\int_{|t|<1/\tau}\Big|\sum\limits_{N_1<p<N_2,p\equiv l_0\bmod q}e(z p^k)\Big|^{m-1}\mathrm{d}z
\end{align*}
$S_1$ and $S_2$ come from those two terms above. We estimate them separately. By Lemma 2.3 (which holds unconditionally when $k=2$; see \cite{BDH}) and Lemma 2.2, using the trivial bound inside the integral, we obtain
\begin{align*}
S_1\ll&N^{\theta(m-1)}\sum_{q<Q}\sum_{l<q,(l,q)=1}N^{\frac{1}{2}-\theta}q^{m/2}\varphi^{-m}(q)\tau^{-1}\sum_{N_1<u<N_2}\max_{(l,q)=1}|E(u,N_2;q,l)|\\
=&N^{\theta(m-1)+\frac{1}{2}-\theta}\sum_{q<Q}q^{m/2}\varphi^{1-m}(q)\tau^{-1}\sum_{N_1<u<N_2}\max_{(l,q)=1}|E(u,N_2;q,l)|\\
\leq&\frac{N^{\theta(m-1)+\frac{1}{2}-\theta}}{\tau}\bigg(\sum_{q<Q}q^{m}\varphi^{2-2m}(q)\bigg)^{1/2}\times\bigg(\sum_{q<Q}\sum_{N_1<u<N_2}\max_{(l,q)=1}|E(u,N_2;q,l)|\Big|^2\bigg)^{1/2}\\
\ll&\frac{N^{\theta(m-1)+\frac{1}{2}-\theta}}{\tau}\times\bigg(\sum_{q<Q}\sum_{N_1<u<N_2}\max_{(l,q)=1}|E(u,N_2;q,l)|\Big|^2\bigg)^{1/2}\\
\ll&\frac{N^{\theta(m-1)+\frac{1}{2}}}{\tau}Q^{1/2}=N^{\theta(m-1)+1-k-\theta+\frac12(\theta-\frac12)+\epsilon}
\end{align*}
Since the coefficient of $m$ is less than $\theta$ and the coefficient of
$k$ is at most $1$, $S_1$ is negligible compared with the expected
main term. For $S_2$ we have
\begin{align*}
S_2\ll&N^{\theta(m-1)}\sum_{q<Q}\sum_{l<q,(l,q)=1}q^{m/2}\varphi^{1-m}(q)\tau^{-1}\max_{N_1<u<N_2}\max_{(l,q)=1}|E(u,N_2;q,l)|\\
\leq&N^{\theta(m-1)}\tau^{-1}\bigg(\sum_{q<Q}q^{m}\varphi^{2-2m}(q)\bigg)^{1/2}\times\bigg(\sum_{q<Q}\max_{N_1<u<N_2}\max_{(l,q)=1}|E(u,N_2;q,l)|\Big|^2\bigg)^{1/2}\\
\leq&\frac{N^{\theta(m-1)+\frac{1}{2}}}{\tau}Q^{1/2}=N^{\theta(m-1)-k+1-\theta+\frac12(\theta-\frac12)+\epsilon}
\end{align*}
This proves the lemma.
\end{proof}
\subsection{The expansion of the major-arc integral}
We shall use $h$ or $a$ interchangeably with $l$. In view of Theorem 1.2, our first goal in this section is to prove
\begin{align*}
I_{11}=\int_{\mathfrak{M}_1}S_p^{m}(t)e(-tN_0)\mathrm{d}t
\sim&\mathfrak{G}(Q)\int_{-1/\tau}^{1/\tau}\left(\sum\limits_{{N_1}<u<{N_2}}\frac{e(zu^k)}{\log{N}}\right)^{m}e(-zN_0)\mathrm{d}z\\
&+O(N^{(2\theta-1)m}(\log N)^{-2m})+O(N^{m/2}Q^{m/2})
\end{align*}
The second error term is estimated in Lemma 4.1:
\[
N^{m/2}Q^{m/2}=N^{(\theta/2+1/4)m}
\]
At the end of Section 4 we shall prove 
\[
\int_{\mathfrak{M}_1}S_p^{m}(t)e(-tN_0)\mathrm{d}t\gg^*N^{\theta(m-1)+1-k-\epsilon}.
\]
Thus the error term from Lemma 4.1 does not affect our main term. We have
\begin{align*}
I_{11}=&\sum\limits_{q<Q}\sum\limits_{1\le a\le q,(a,q)=1}\int_{|t-\frac{a}{q}|<\frac{1}{\tau}}\left(\sum\limits_{N_1<p<N_2}e(tp^k)\right)^{m}e(-N_0t)\mathrm{d}t\\
=&\sum\limits_{q<Q}\sum\limits_{1\le a\le q,(a,q)=1}\int_{-1/\tau}^{1/\tau}\left(\sum\limits_{N_1<p<N_2}e((z+a/q)p^k)\right)^{m}e(-N_0(z+a/q))\mathrm{d}z\\
\sim&\sum\limits_{q<Q}\sum\limits_{1\le a\le q,(a,q)=1}C_q^{m}(a,k)e(-\frac{aN_0}{q})\int_{-1/\tau}^{1/\tau}
\left\{
\sum\limits_{N_1<u<N_2}\frac{e(zu^k)}{\varphi(q)\log{u}}+\int_{N_1}^{N_2}e(zu^k)\mathrm{d}E(u;q,h_0)
\right\}
^{m}e(-N_0z)\mathrm{d}z.
\end{align*}
We expand the expression inside the braces; the cross terms will be shown to be negligible compared with the main term. We have
\begin{align}
&\left\{
\sum\limits_{N_1<u<N_2}\frac{e(zu^k)}{\varphi(q)\log{u}}+\int_{N_1}^{N_2}e(zu^k)\mathrm{d}E(u_0;q,h_0)
\right\}^{m}\notag\\
=&\left(\sum\limits_{N_1<u<N_2}\frac{e(zu^k)}{\varphi(q)\log{u}}\right)^{m}+\sum\limits_{1\leq r\leq m}\binom{m}{r}\left(\sum\limits_{N_1<u<N_2}\frac{e(zu^k)}{\varphi(q)\log{N}}\right)^{m-r}\left(\int_{N_1}^{N_2}e(zu^k)\mathrm{d}E(u;q,h_0)\right)^r.\label{eq:binomial}
\end{align}
Since
\begin{equation}
\int_{N_1}^{N_2}e(zu^k)\mathrm{d}E^r(u;q,h_0)=E^r(u;q,h_0)e(zu^k)\bigg|_{N_1}^{N_2}-\int_{N_1}^{N_2}E^r(u;q,h_0)\mathrm{d}e(zu^k).
\end{equation}
Then, for the real part, we have
\[
\Re\int_{N_1}^{N_2}e(zu^k)\mathrm{d}E^r(u;q,h_0)=\Re E^r(u;q,h_0)e(zu^k)\bigg|_{N_1}^{N_2}-\Re\int_{N_1}^{N_2}E^r(u;q,h_0)\mathrm{d}e(zu^k)
\]
Since $\tau>N_2^k-N_1^k$, the interval $[N_1,N_2]$ can be partitioned into two intervals $[N_1,M]$ and $[M,N_2]$ on which the sign is the same. Then, by the first mean value theorem, there exist $t_1\in[N_1,M]$ and $t_2\in[M,N_2]$ such that
\begin{align*}
\Re\int_{N_1}^{N_2}E^r(u;q,h_0)\mathrm{d}e(zu^k)=&\Re\int_{N_1}^{M}E^r(u;q,h_0)\mathrm{d}e(zu^k)+\Re\int_{M}^{N_2}E^r(u;q,h_0)\mathrm{d}e(zu^k)\\
=&E^r(u_1;q,h_0)\Re\int_{N_1}^{M}\mathrm{d}e(zu^k)+E^r(u_2;q,h_0)\Re\int_{M}^{N_2}\mathrm{d}e(zu^k)
\end{align*}
Thus
\[
\Re\int_{N_1}^{N_2}E^r(u_0;q,h_0)\mathrm{d}e(zu^k)\ll\max_{N_1\leq t_1\leq t_2}E^r(t;q,h_0)
\]
For the imaginary part, we also have
\[
\Im\int_{N_1}^{N_2}e(\frac{bf(t)}{d})\mathrm{d}E^r(u;q,r)\ll\max_{N_1\leq t\leq N_2}E^r(t;q,r)
\]
Then
\[
\left|\int_{N_1}^{N_2}e(zu^k)\mathrm{d}E^r(u;q,h_0)\right|\ll\max_{N_1\leq t\leq N_2}E^r(t;q,r)= E^r(u_0;q,r)
\]
Substituting this into \eqref{eq:binomial}, we obtain
\begin{align*}
&\left\{
\sum\limits_{N_1<u<N_2}\frac{e(zu^k)}{\varphi(q)\log{u}}+\int_{N_1}^{N_2}e(zu^k)\mathrm{d}E^r(u_0;q,h_0)
\right\}^{m}
-\left(\sum\limits_{N_1<u<N_2}\frac{e(zu^k)}{\varphi(q)\log{u}}\right)^{m}\\
\ll&\sum\limits_{1\leq r\leq m}\binom{m}{r}\left(\sum\limits_{N_1<u<N_2}\frac{e(zu^k)}{\varphi(q)\log{u}}\right)^{m-r}E^{r}(u_0;q,l)
\end{align*}
Returning to $I_{11}$, we have
\begin{align*}
&I_{111}^*+I_{112}^*\\
=&\sum\limits_{q<Q}\sum\limits_{h<q,(h,q)=1}C_q^{m}(h,k)e(-\frac{aN_0}{q})\\
&\int_{-1/\tau}^{1/\tau}
\left\{
\sum\limits_{N_1<u<N_2}\frac{e(zu^k)}{\varphi(q)\log{u}}+\int_{N_1}^{N_2}e(zu^k)\mathrm{d}E^r(u_0;q,l)
\right\}e(-N_0z)\mathrm{d}z
\end{align*}
where
\[
I_{111}^*=\mathfrak{G}(Q)\int_{-1/\tau}^{1/\tau}\left(\sum\limits_{{N_1}<u<{N_2}}\frac{e(zu^k)}{\log{u}}\right)^{m}e(-zN_0)\mathrm{d}z
\]
and
\begin{align*}
I_{112}^*\ll&\sum\limits_{q<Q}\sum\limits_{h<q,(h,q)=1}C_q^{m}(h,k)e(-\frac{aN_0}{q})\int_{-1/\tau}^{1/\tau}
\sum\limits_{1\leq r\leq m}\binom{m}{r}\left(\sum\limits_{N_1<u<N_2}
\frac{e(zu^k)}{\varphi(q)\log{u}}\right)^{m-r}\\&\cdot E^r(u_0,q,h_0)
\mathrm{d}z.
\end{align*}
For $I_{111}^*$, we have
\begin{align*}
I_{111}^*&=\mathfrak{G}(Q)\int_{-1/\tau}^{1/\tau}\left(\sum\limits_{{N_1}<u<{N_2}}\frac{e(zu^k)}{\log{u}}\right)^{m}e(-zN_0)\mathrm{d}z\\
&\sim\mathfrak{G}(Q)\int_{-1/\tau}^{1/\tau}\left(\sum\limits_{{N_1}<u<{N_2}}\frac{e(zu^k)}{\log{N}}\right)^{m}e(-zN_0)\mathrm{d}z.
\end{align*}
For $I_{112}^*$, we have
\begin{align*}
I_{112}^*=&
\sum\limits_{1\leq r\leq m}\binom{m}{r}\sum\limits_{q<Q}\sum\limits_{h<q,(h,q)=1}C_q^{m}(h,k)e(-\frac{aN_0}{q})\int_{-1/\tau}^{1/\tau}\left(\sum\limits_{N_1<u<N_2}\frac{e(zu^k)}{\varphi(q)\log{u}}\right)^{m-r}\\[2pt]
&\cdot E^r(u_0;q,h)2^r\mathrm{d}z\\[2pt]
\ll&_{m}
\sum\limits_{1\leq r\leq m}\sum\limits_{q<Q}\sum\limits_{h<q,(h,q)=1}\left|C_q^{m}(h,k)(\varphi(q)\log{N})^{r-m}\right|\\[2pt]
&\cdot\int_{-1/\tau}^{1/\tau}\left|\sum\limits_{N_1<u<N_2}e(zu^k)\right|^{m-r}
|E^r(u_0;q,h_0)|\mathrm{d}z.
\end{align*}
Therefore
\begin{align*}
I_{112}^*&\ll \max\limits_r\sum\limits_{q<Q}\sum\limits_{(a,q)=1,a\leq q}|C_q^{m}(a,k)|(\varphi(q)\log{N})^{r-m}\int_{-1/\tau}^{1/\tau}\left|\sum\limits_{N_1<u<N_2}e(zu^k)\right|^{m-r}\\[2pt]
&\cdot\mathrm{d}z\,E^r(u_0;q,h_0).
\end{align*}
It remains to estimate each part and then take the maximum of the resulting bounds. Then
\begin{align*}
I_{112}^*&\ll\bigg(\sum\limits_{q<Q}\sum\limits_{(a,q)=1,a\leq q}|C_q^{mw_1}(a,k)\varphi^{(r-m)w_1}(q)|\bigg)^{1/w_1}\\
&\times\left(\sum\limits_{q<Q}\sum\limits_{(a,q)=1,a\leq q}\left(\int_{-1/\tau}^{1/\tau}\left|\sum\limits_{N_1<u<N_2}e(zu^k)\right|^{m-r}\mathrm{d}z\right)^{w_2}\right)^{1/w_2}\\
&\times\left(\sum\limits_{q<Q}\sum\limits_{(a,q)=1,a\leq q}|E^{rw_3}(u_0;q,h_0)|\right)^{1/w_3},
\end{align*}
where we have used Hölder's inequality once more, with $w_1^{-1}+w_2^{-1}+w_3^{-1}=1$, and $w_i>1,i=1,2,3$.  Here $I_{112}^*\ll V_1^{1/w_1}V_2^{1/w_2}V_3^{1/w_3}$, where
\begin{align*}
\qquad&V_1=\left|\sum\limits_{q<Q}\sum\limits_{(a,q)=1,a\leq q}|C_q^{mw_1}(a,k)\varphi^{(r-m)w_1}(q)|\right|\\
&V_2=\sum\limits_{q<Q}\sum\limits_{(a,q)=1,a\leq q}\left(\int_{-1/\tau}^{1/\tau}\left|\sum\limits_{N_1<u<N_2}e(zu^k)\right|^{m-r}\mathrm{d}z\right)^{w_2}\\
&V_3=\sum\limits_{q<Q}\sum\limits_{(a,q)=1,a\leq q}|E^{rw_3}(u_0;q,h_0)|
\end{align*}
We now estimate these terms one by one. By $\varphi(q)\gg q(\log{\log{q}})^{-1}$, we have
\begin{align*}
V_1\ll&\left|\sum\limits_{q<Q}\sum\limits_{(a,q)=1,a\leq q}q^{\frac{1}{2}mw_1}\varphi^{w_1(r-m)}(q)\right|\\
\ll&\sum\limits_{q<Q}\sum\limits_{(a,q)=1,a\leq q}q^{mw_1/2+w_1(r-m)}(\log{\log{q}})^{-w_1(r-m)}\\
\ll&\max\left\{Q^{-\frac{mw_1}{2}+w_1r+2}(\log{\log{Q}})^{-w_1(r-m)},1\right\}.
\end{align*}
By the estimate for the main term in \cite{wei2014sumspowersequalprimes}, we have
\[
\int_{0<z<1}\left|\sum\limits_{N_1<u<N_2}e(zu^k)\right|^{m-r}\mathrm{d}z\asymp N^{(m-1)\theta+1-k}
\]
Then
\begin{align*}
\qquad V_2\ll&\sum\limits_{q<Q}\sum\limits_{(a,q)=1,a\leq q}\left(\frac{2}{\tau}\right)^{w_2}\left(\int_{-1/2<z<1/2}\left|\sum\limits_{N_1<u<N_2}e(zu^k)\right|^{m-r}\mathrm{d}z\right)^{w_2}\\
=&\sum\limits_{q<Q}\sum\limits_{(a,q)=1,a\leq q}\left(\frac{2}{\tau}\right)^{w_2}\left(\int_{0<z<1}\left|\sum\limits_{N_1<u<N_2}e(zu^k)\right|^{m-r}\mathrm{d}z\right)^{w_2}\\
\ll&\sum\limits_{q<Q}q\left(\frac{2}{\tau}\right)^{w_2}N^{\epsilon}N^{((m-1-r)\theta+1-k)w_2}\\
\ll&Q^2\tau^{-w_2}N^{((m-1-r)\theta+1-k)w_2+\epsilon}
\end{align*}
By Lemma 2.3,
\begin{align*}
\qquad V_3\ll& \sum\limits_{q<Q}\sum\limits_{(a,q)=1,a\leq q}\left(\int_{-1/\tau}^{1/\tau}\max\limits_{u}\left|E(u;q,h_0)\right|^{r}\mathrm{d}z\right)^{w_3}\\
\ll&\sum\limits_{q<Q}\varphi(q)\left(\frac{2}{\tau}\right)^{w_3}\max\limits_{u}\left|E(u;q,h_0)\right|^{w_3r}\ll Q^2\tau^{-w_3}N^{w_3r/2}.
\end{align*}
Combining these estimates, in the first case $r>m/2$, we obtain
\begin{align*}
\qquad I_{112}^*\ll& Q^{-\frac{m}{2}+r+2(w_1^{-1}+w_2^{-1}+w_3^{-1})}\tau^{-2}N^{r/2}N^{(m-1-r)\theta+1-k+\epsilon}\\
\ll&Q^{-\frac{m}{2}+r+2}\tau^{-2}N^{r/2}N^{(m-1-r)\theta+1-k+\epsilon}\\
\ll&N^{(-\frac{m}{2}+r+2)g(\theta)+\frac{r}{2}-2k-2\theta+2+\epsilon}N^{(m-1-r)\theta+1-k+\epsilon}
\end{align*}
where $g(\theta)=\theta-\frac{1}{2}$. Under this condition the coefficient of $r$ is equal to zero. Thus, for any given $\epsilon>0$, we have
\begin{align*}
I_{112}^*\ll&N^{(-\frac{m}{2}+2)(\theta-\frac{1}{2})-2k-2\theta+2+\epsilon}N^{(m-1)\theta+1-k+\epsilon}=o(N^{\theta (m-1)+1-k})
\end{align*}
When $r<m/2$, the estimate for $V_1$ is again $O(1)$. In this second case, we have
\begin{align*}
I_{112}^*\ll& V_2^{1/w_2}V_3^{1/w_3}\\
\ll&Q^{2(w_2^{-1}+w_3^{-1})+\epsilon}\tau^{-2}N^{r/2}N^{(m-1-r)\theta+1-k+\epsilon}\\
\ll&N^{(m-1)\theta+(\frac{1}{2}-\theta)r-3k-3\theta+5+\epsilon}
\end{align*}
The expression above is decreasing in $r$, so its maximum occurs at $r=1$. We have
\[
I_{112}^*\ll N^{(m-1)\theta-3k-4\theta+11/2+\epsilon}=o(N^{\theta(m-1)+1-k})
\]
Thus we prove for any given $r$, $V_1^{1/w_1}V_2^{1/w_2}V_3^{1/w_3}=o(N^{\theta (m-1)+1-k})$. Substituting this estimate into $I_{11}$ completes the proof.

\subsection{The singular series and local densities}
We now turn to the computation of the integral
\[
I_{11}=\int_{-1/\tau}^{1/\tau}\left(\sum\limits_{{N_1}<u<{N_2}}{e(zu^k)}\right)^{m}e(-zN_0)\mathrm{d}z
\]
and the series
\[
\mathfrak{G}(Q)=\sum\limits_{q<Q}\sum\limits_{(a,q)=1,a\leq q}C_q^{m}(a,k)e(-\frac{amN^k}{q})\varphi^{-m}(q).
\]
\begin{lemma}
We set
\[
\mathfrak{G}(Q)=\sum\limits_{q<Q}\sum\limits_{(a,q)=1,a\leq q}C_q^{m}(a,k)e(-\frac{amN^k}{q})\varphi^{-m}(q).
\]
Then there exist two constants $c_1,c_2$ independent of $Q$ such that $c_1<|\mathfrak{G}(Q)|<c_2$.
\end{lemma}

For completeness, we give a short proof. By the standard estimate $|C_q(a,k)|\ll q^{1/2}$,
\begin{align*}
|\mathfrak{G}(Q)|\le &\sum\limits_{q<Q}\sum\limits_{(a,q)=1,a\leq q}\left|C_q^{m}(k)\varphi^{-m}(q)\right|\\
\ll&\sum\limits_{q=1}^\infty\left|q^{1/2}\varphi^{1-m}(q)\right|\ll \prod\limits_{p}\left(1-\frac{p^{1/2}}{(p-1)^{m-1}}\right)^{-1}.
\end{align*}
For $p\geq 3$ we have $p-1>p^{1/2}$, and since $\frac{m}{2}-1>c_3>1$, we obtain
\[
\prod\limits_{p}\left(1-\frac{p^{1/2}}{(p-1)^{m-1}}\right)^{-1}\leq\prod\limits_{p}\left(1-p^{1-\frac{m}{2}}\right)^{-1}=\zeta\left(\frac{m}{2}-1\right)<c_2.
\]
Thus $|\mathfrak{G}(Q)|<c_2$. We now estimate $\mathfrak{G}(Q)$ from below, assuming the congruence condition. By the Chinese remainder theorem, the singular series factors as
\[
\lim\limits_{Q\rightarrow\infty}\mathfrak{G}(Q)=\prod\limits_{p}\sum\limits_{j=0}^{\infty}\sum\limits_{(a,p^j)=1}C_{p^j}^{m}(a,k)e(-\frac{amN^k}{p^j})\varphi^{-m}(p^j).
\]
The singular series is a product of local densities. We define
\begin{align*}
\sigma_{p}=&\sum\limits_{j=0}^{\infty}\sum\limits_{(a,p^j)=1}C_{p^j}^{m}(a,k)e(-\frac{amN^k}{p^j})\varphi^{-m}(p^j)\\
=&\sum_{j=0}^\infty\varphi^{-m}(p^j)\sum_{\substack{0\leq x_1,\dots,x_{m}\leq p^j-1\\ p\nmid x_1\cdots x_{m}}}c_{p^j}(x_1^k+\cdots+x_{m}^k-N_0)
\end{align*}
where $c_{p^j}(x_1^k+\cdots+x_{m}^k-N_0)$ is the Ramanujan sum and $N_{p^j}(mN^k)$ counts the tuples of units,
\[
N_{p^j}(mN^k)=\#\left\{(x_1,\dots,x_{m})\in(\mathbf Z/p^j\mathbf Z)^{m}:
p\nmid x_1\cdots x_{m},\ \sum_{i\le m}x_i^k\equiv mN^k\bmod{p^j}\right\}.
\]
For $j\ge2$, let $N_{p^{j-1}}^*$ denote the number of unit lifts of the solutions modulo $p^{j-1}$ to $\mathbf Z/p^j\mathbf Z$; then $N_{p^{j-1}}^*=p^{m}N_{p^{j-1}}(mN^k)$, and
\begin{align*}
&\sum_{\substack{0\le x_1,\dots,x_{m}\le p^j-1\\ p\nmid x_1\cdots x_{m}}}c_{p^j}(x_1^k+\cdots+x_{m}^k-N_0)\\
=&(p^j-p^{j-1})N_{p^j}(mN^k)-p^{j-1}\bigl(N_{p^{j-1}}^*(mN^k)-N_{p^j}(mN^k)\bigr).
\end{align*}
For $j=1$, the same computation gives
\[
\sum_{\substack{0\le x_1,\dots,x_{m}\le p-1\\ p\nmid x_1\cdots x_{m}}}c_p\Bigl(\sum_{i\le m}x_i^k-N_0\Bigr)
=pN_p(mN^k)-(p-1)^{m}.
\]
Consequently, with $N_{p^0}^*:=(p-1)^{m}$,
\[
\sigma_p=1+\sum_{j=1}^\infty\varphi^{-m}(p^j)\Bigl(p^jN_{p^j}(mN^k)-p^{j-1}N_{p^{j-1}}^*(mN^k)\Bigr),
\]
where $N_{p^{j-1}}^*=p^{m}N_{p^{j-1}}(mN^k)$ for $j\ge2$.
We shall need the following lemma.
\begin{lemma}[Local density]\label{lem:local-density}
Let $k\ge3$ be an odd prime, let $m>2k+3$, and put $N_0=mN^k$.
Assume the congruence condition
\[
N_0\equiv \sum\limits_{i\leq m}x_i^k \bmod{p^j}
\]
where
\[
N_{p^j}(N_0)
=\{(x_1,\dots,x_{m})\in(\mathbf{Z}/p^j\mathbf{Z})^{m}:
p\nmid x_1\cdots x_{m},\
\sum_{i=1}^{m} x_i^k\equiv N_0\bmod{p^j}\}.
\]
Then the local density
\[
\beta_p:=\lim_{j\to\infty}\frac{N_{p^j}(mN^k)}{p^{j(m-1)}}
\]
is positive for every prime $p$. Moreover, if $k<p$ and $m=O(k)$, then
\[
\beta_p=1+O_k(p^{-1}).
\]
\end{lemma}

\begin{proof}
First suppose that $p\nmid k$. We exhibit a nonsingular solution
modulo $p$. Consider the gradient
\[
\nabla\Bigl(\sum_{i=1}^{m} x_i^k-N_0\Bigr)
=\bigl(kx_1^{k-1},\dots,kx_{m}^{k-1}\bigr)
\]
For fixed $n$, we consider lifting the solution $(x_1,\dots,x_{m})$ modulo $p^n$ to modulo $p^{n+1}$. Set $y_i=x_i+t_ip^n$, and denote the sum of $k$-th powers of $x_i$ by $F(x_1,\dots,x_{m})$; then
\[
F(y_1,\dots,y_{m})\equiv F(x_1,\dots,x_{m})+p^n\sum_{i\leq m}t_i\frac{\partial}{\partial x_i}F(x_1,\dots,x_{m})\bmod p^{n+1}.
\]
For any fixed $(t_1,\dots,t_{m-1})$, the value of $t_{m}$ is determined
modulo $p$. In this way we obtain a solution modulo $p$ at which the
gradient does not vanish. By Hensel's lemma, this solution lifts to $p^{(j-1)(m-1)}$ solutions
modulo $p^j$ for every $j\ge1$. In particular,
\[
N_{p^j}(N_0)\ge p^{(j-1)(m-1)}.
\]
Every solution modulo $p$ has all coordinates prime to $p$, and
$p\nmid k$, so the gradient is nonzero at every solution. Therefore
\[
N_{p^j}(N_0)=p^{(j-1)(m-1)}N_p(N_0)
\]
for all $j\ge1$. It remains to estimate $N_p(N_0)$.
For $t\ne0$, the inner sum is a Gauss sum over the subgroup of
$k$-th powers. To estimate it, put
\[
H=\{x^k:\ x\in(\mathbf Z/p\mathbf Z)^\times\},\qquad
|H|=\frac{p-1}{d},\quad d=\gcd(k,p-1).
\]
Each $y\in H$ has exactly $d$ preimages in
$(\mathbf Z/p\mathbf Z)^\times$, so
\[
\sum_{x=1}^{p-1}e\Bigl(\frac{tx^k}{p}\Bigr)
=d\sum_{y\in H}e\Bigl(\frac{ty}{p}\Bigr).
\]
Let $\chi_0,\dots,\chi_{d-1}$ be the characters of
$(\mathbf Z/p\mathbf Z)^\times/H$, so that
\[
1_H(y)=\frac1d\sum_{j=0}^{d-1}\chi_j(y)\qquad(y\ne0).
\]
Then
\[
\sum_{y\in H}e\Bigl(\frac{ty}{p}\Bigr)
=\frac1d\sum_{j=1}^{d-1}\tau(\chi_j),
\qquad
\tau(\chi_j)=\sum_{y\in(\mathbf Z/p\mathbf Z)^\times}
\chi_j(y)e\Bigl(\frac{ty}{p}\Bigr).
\]
For $j\ne0$, the character $\chi_j$ is nonprincipal and $t\ne0$,
so the standard Gauss sum evaluation gives
\[
|\tau(\chi_j)|=\sqrt p.
\]
Hence
\[
\Bigl|\sum_{x=1}^{p-1}e\Bigl(\frac{tx^k}{p}\Bigr)\Bigr|
\le(d-1)\sqrt p\le k\sqrt p.
\]
Therefore the total contribution of $t\ne0$ is at most
\[
\frac1p\sum_{t=1}^{p-1}
\Bigl|\sum_{x=1}^{p-1}e\Bigl(\frac{tx^k}{p}\Bigr)\Bigr|^m
\le (k\sqrt p)^m=O_k(p^{m/2}).
\]
Thus
\[
\beta_p=\frac{N_{p^j}(mN^k)}{p^{j(m-1)}}=\Bigl(1-\frac1p\Bigr)^{m}+O_k(p^{1-m/2})
\]
This proves the lemma.
\end{proof}
For $p\mid k$, the primality of $k$ forces $p=k$. Let
\[
H=\{\omega^p\bmod p^2:\ 1\le\omega\le p-1\}\subset
(\mathbf Z/p^2\mathbf Z)^\times
\]
be the Teichmüller group; it is the set of $(p-1)$-th roots of unity
modulo $p^2$, and its reduction modulo $p$ is $\mathbf F_p^\times$.
Every unit $x$ modulo $p^j$ has a unique representation
\[
x\equiv\omega u\pmod{p^j},\qquad \omega\in H,\quad
u\in1+p\mathbf Z/p^j\mathbf Z.
\]
Then $x^k=\omega u^k$, and $u^k\in1+p^2\mathbf Z/p^j\mathbf Z$.

The group $1+p\mathbf Z/p^j\mathbf Z$ is cyclic of order
$p^{j-1}$. Hence for $j\ge2$ the map $u\mapsto u^k$ from
$1+p\mathbf Z/p^j\mathbf Z$ to $1+p^2\mathbf Z/p^j\mathbf Z$ has
kernel of size $p$, and every image has exactly $p$ preimages.
Fix an $\omega$-tuple $(\omega_1,\dots,\omega_m)\in H^m$. The congruence
\[
\sum_{i\le m}\omega_i u_i^k\equiv N_0\pmod{p^j}
\]
implies, modulo $p^2$,
\[
\sum_{i\le m}\omega_i\equiv N_0\pmod{p^2}.
\]
If this holds, write $u_i^k=1+p^2v_i$. Then the equation becomes
\[
\sum_{i\le m}\omega_i v_i
\equiv-\frac{\sum_{i\le m}\omega_i-N_0}{p^2}
\pmod{p^{j-2}}.
\]
This is a linear equation in the $v_i$; for any choice of
$v_1,\dots,v_{m-1}$, the last variable $v_m$ is determined modulo
$p^{j-2}$, since $\omega_m$ is a unit. Hence this $\omega$-tuple
contributes
\[
p^{m}\cdot p^{(j-2)(m-1)}=p^{j(m-1)-m+2}
\]
solutions $x$. Therefore
\[
N_{p^j}(N_0)=R_{p^2}(N_0)\,p^{j(m-1)-m+2},
\]
where $R_{p^2}(N_0)$ is the number of $\omega$-tuples in $H^m$ with
$\sum_{i\le m}\omega_i\equiv N_0\pmod{p^2}$.

It remains to show $R_{p^2}(N_0)>0$. The reduction of $H$ modulo $p$
is $\mathbf F_p^\times$, so the reduction of the $m$-fold sumset $mH$
is $\mathbf Z/p\mathbf Z$ for $m\ge2$. Let $K$ be the stabilizer of
$mH$ in $\mathbf Z/p^2\mathbf Z$. By Kneser's theorem,
\[
|mH|\ge m|H+K|-(m-1)|K|.
\]
If $K=\{0\}$, then
\[
|mH|\ge m(p-1)-m+1=(p-2)m+1>p^2
\]
for $m>2p+3$, so $mH=\mathbf Z/p^2\mathbf Z$. If $K\ne\{0\}$, then
$K$ contains $p\mathbf Z/p^2\mathbf Z$ or is the whole group; in
either case $mH$ is a union of cosets of $p\mathbf Z/p^2\mathbf Z$,
and since its reduction is all of $\mathbf Z/p\mathbf Z$, it equals
$\mathbf Z/p^2\mathbf Z$. Thus $R_{p^2}(N_0)>0$, and
\[
\beta_p=\frac{N_{p^j}(N_0)}{p^{j(m-1)}}
=R_{p^2}(N_0)\,p^{-m+2}>0.
\]
This completes the proof for $p\mid k$.
It remains to check $p=2$. The map $x\mapsto x^k$ is an automorphism of $(\mathbf{Z}/2^n\mathbf{Z})^\times$ for every $n\ge1$. Hence the local density modulo $2^n$ equals
\[
2^{-n(m-1)}\#\left\{(x_1,\dots,x_{m})\in(\mathbf{Z}/2^n\mathbf{Z})^{m}:
x_1+\cdots+x_{m}\equiv N_0\bmod{2^n},\ (x_i,2)=1\right\}.
\]
There are $2^{(n-1)(m-1)}$ choices for the first $m-1$ units, and the last variable is then determined modulo $2^n$; it is automatically a unit since $N_0-\sum_{i<m}x_i$ is odd. Thus the above quantity equals $2^{-(m-1)}$, which is positive.

We now return to Lemma 4.2. For $p\nmid k$, the lifting identity in Lemma 4.3 gives
$N_{p^j}=p^{(j-1)(m-1)}N_p$ for every $j\ge1$. Substituting this into the Ramanujan-sum expression, we obtain
\[
\sum_{\substack{x_1,\dots,x_{m}\\ p\nmid x_1\cdots x_{m}}}c_{p^j}\Bigl(\sum_{i=1}^{m}x_i^k-N_0\Bigr)=0
\qquad (j\ge2).
\]
Thus, for $p\nmid k$, only $j=0,1$ contribute, and
\[
\sigma_p=\frac{pN_p}{(p-1)^{m}}=1+O(p^{1-m/2}).
\]
For a general prime $p$, set
\[
A_j=\frac{N_{p^j}(mN^k)}{p^{j(m-1)}}\qquad(j\ge1),
\]
so that $\beta_p=\lim_{j\to\infty}A_j$. Since every unit tuple modulo $p^{j-1}$
has exactly $p^{m}$ unit lifts modulo $p^j$, one has
$N^*_{p^{j-1}}=p^{m}N_{p^{j-1}}$ for $j\ge2$, while
$N^*_{p^0}=(p-1)^{m}$. Substituting into the expression for $\sigma_p$ gives
\[
\begin{aligned}
\sigma_p
&=1+\frac{p^{m}A_1-(p-1)^{m}}{(p-1)^{m}}
+\sum_{j\ge2}(1-1/p)^{-m}(A_j-A_{j-1})\\
&=(1-1/p)^{-m}\beta_p.
\end{aligned}
\]
Since $\beta_p>0$ by Lemma 4.3, the local factor $\sigma_p$ is nonzero for
every prime $p$, including $p\mid k$ and $p=2$. Since $\sum_p p^{1-m/2}$
converges, the infinite product converges; hence the singular series, being a
product of nonzero local factors, does not vanish. This completes the proof of
Lemma 4.2.

\subsection{The singular integral and the tail estimate}
Our next goal is to prove that
\[
\int_{-1/\tau}^{1/\tau}\left(\sum\limits_{{N_1}<u<{N_2}}{e(zu^k)}\right)^{m}e(-zN_0)\mathrm{d}z\sim \int_{0}^{1}\left(\sum\limits_{{N_1}<u<{N_2}}{e(zu^k)}\right)^{m}e(-zN_0)\mathrm{d}z.
\]
By the periodicity of the triangular integral, it suffices to prove that
\[
\int_{1/\tau}^{\frac{1}{2}}\left(\sum\limits_{{N_1}<u<{N_2}}{e(zu^k)}\right)^{m}e(-zN_0)\mathrm{d}z=o\left(\int_{0}^{1}\left(\sum\limits_{{N_1}<u<{N_2}}{e(zu^k)}\right)^{m}e(-zN_0)\mathrm{d}z\right).
\]
By \cite{wei2014sumspowersequalprimes},
\[
\int_{0}^{1}\left(\sum\limits_{{N_1}<u<{N_2}}{e(zu^k)}\right)^{m}e(-zN_0)\mathrm{d}z \asymp N^{\theta(m-1)+1-k}.
\]
It remains to estimate the following integral:
\[
\int_{1/\tau}^{1/2}\left(\sum\limits_{{N_1}<u<{N_2}}{e(zu^k)}\right)^{m}e(-zN_0)\mathrm{d}z 
\]
Since $kzu^{k-1}<kN^{k-1}\tau^{-1}$, and since $\theta<1$, \cite{wei2014sumspowersequalprimes} gives
\[
\sum\limits_{{N_1}<u<{N_2}}{e(zu^k)}=k^{-1}\sum_{N_1^k<u<N_2^k}u^{-1+1/k}e(zu)+O(1+zN^{k-1+\theta})
\]
Also, by \cite{wei2014sumspowersequalprimes},
\[
k^{-1}\sum_{N_1^k<u<N_2^k}u^{-1+1/k}e(zu)\ll\frac{N^\theta}{1+N^{k-1+\theta}\|z\|}
\]
Then 
\begin{align*}
\int_{1/\tau}^{1/2}\left(\sum\limits_{{N_1}<u<{N_2}}{e(zu^k)}\right)^{m}e(-zN_0)\mathrm{d}z\ll&\int_{1/\tau}^{1/2}\left(\frac{N^\theta}{1+N^{k-1+\theta}z}\right)^{m}\mathrm{d}z
\\
+&O\left((\frac{N^\theta}{1+N^{k-1+\theta}\|z\|})^{m-1}\right)\\
\end{align*}
Let $s=1+N^{k-1+\theta}z$; then $\mathrm{d}s=N^{k-1+\theta}\mathrm{d}z$, and hence
\begin{align*}
\int_{1/\tau}^{1/2}\left(\frac{N^\theta}{1+N^{k-1+\theta}z}\right)^{m}\mathrm{d}z
=&N^{\theta m+1-k-\theta}\int_{1+N^{k-1+\theta}/\tau}^{1+N^{k-1+\theta}/2}s^{-m}\mathrm{d}z\\
\ll&N^{\theta(m-1)+1-k}(\log \log N)^{-m}
\end{align*}
This completes the proof.
\section{The Minor-Arc Integral}
\subsection{The expansion near rational points}
Since $x=o(q)$, the primality condition can no longer be transformed on the minor arcs. To estimate the exponential sum there, we need the following lemma. In the sequel we suppress the parameter $\epsilon$ in $\tau$, since it does not affect the argument.
\begin{lemma}
Let
\[
I_{20}=\int\limits_{\mathfrak{M_2}}S_p^{m}(t)\mathrm{d}t.
\]
Then
\[
I_{20}\leq\sum\limits_{\tau_0<q<\tau}\sum\limits_{h<q,(h,q)=1}\frac{2}{q\tau^{1/2}}|S_p^{m}(\frac{h}{q})|(1+\frac{N^{3}\log N}{\tau})+o(I_{11})
\]
\end{lemma}
\begin{proof}
The covering of the minor arc shows
\[
\int\limits_{\mathfrak{M_2}}S_p^{m}(t)\mathrm{d}t\leq\sum\limits_{\tau_0<q<\tau}\sum\limits_{h<q,(h,q)=1}\bigg|\int_{|t-h/q|<1/q\tau^{1/2}}S_p^{m}(t)\mathrm{d}t\bigg|
\]
Let
\[
t\in[\frac{h}{q}-\frac{1}{q\tau^{1/2}},\frac{h}{q}+\frac{1}{q\tau^{1/2}}].
\]
Then
\[
S_p(t)-S_p(h/q)=\sum\limits_{N_1<p<N_2}e(p^kt)-e(\frac{h}{q}p^k)
\]
The expansion of $e(p^kt)$ shows 
\[
e(p^kt)-e(\frac{h}{q}p^k)=(t-\frac{h}{q})N^{k}e(\frac{h}{q}p^k)+\frac{1}{2}N^{2k}(t-\frac{h}{q})^2e(\frac{h}{q}p^k)
\]
Substituting this back into the expression above,
\begin{align*}
\sum\limits_{N_1<p<N_2}(e(p^kt)-e(\frac{h}{q}p^k))=&(t-\frac{h}{q})N^{k}\sum\limits_{N_1<p<N_2}e(\frac{h}{q}p^k)\\
+&\frac{1}{2}(t-\frac{h}{q})^2N^{2k}S_p(\frac{h}{q})+\frac{1}{6}(t-\frac{h}{q})^3N^{3k}S_p(\frac{h}{q})\\
+O((t-\frac{h}{q})^4N^{4k}S_p(\frac{h}{q}))
\end{align*}
Hence
\begin{align*}
S_p^{m}(t)-S_p^{m}(h/q)=&(t-\frac{h}{q})N^{k}\sum_{0\leq \ell\leq m-1} S_p^{m-\ell}(h/q)S_p^{\ell}(t)\\
+&\frac{1}{2}(
t-\frac{h}{q})^2N^{2k}\sum_{0\leq \ell\leq m-1} S_p^{m-\ell}(h/q)S_p^{\ell}(t)\\
+&\frac{1}{6}(t-\frac{h}{q})^3N^{3k}\sum_{0\leq \ell\leq m-1}S_p^{m-\ell}(h/q)S_p^{\ell}(t)\\
+&O((t-\frac{h}{q})^4N^{4k}S_p(\frac{h}{q}))
\end{align*}
Let
\[
S_p^{\ell}(t)=S_p^{\ell}(t)-S_p^{\ell}(\frac{h}{q})+S_p^{\ell}(\frac{h}{q})
\]
Substituting this into the sum above and using the symmetry of the integral, we obtain the following identity; the powers of $N^k$ can then be raised:
\begin{align*}
S_p^{m}(t)-S_p^{m}(h/q)=&m(t-\frac{h}{q})N^{k}S_p^{m}(h/q)\\
+&\frac{m}{2}(t-\frac{h}{q})^2N^{2k}S_p^{m}(h/q)\\
+&(t-\frac{h}{q})N^{k}\sum_{1\leq \ell\leq m-1}S_p^{m-\ell}(h/q)(S_p^{\ell}(t)-S_p^{\ell}(\frac{h}{q}))\\
+&\frac{1}{2}(t-\frac{h}{q})^2N^{2k}\sum_{1\leq \ell\leq m-1}S_p^{m-\ell}(h/q)(S_p^{\ell}(t)-S_p^{\ell}(\frac{h}{q}))\\
+&\frac{1}{6}(t-\frac{h}{q})^3N^{3k}S_p^{m-1}(h/q)\\
+&\frac{1}{6}(t-\frac{h}{q})^3N^{3k}\sum_{0\leq \ell\leq m-1}S_p^{m-\ell}(h/q)(S_p^{\ell}(t)-S_p^{\ell}(h/q))\\
+&O((t-\frac{h}{q})^4N^{4k+\theta m})
\end{align*}
Since the last sum is similar to the preceding one, the identity can be rewritten as
\begin{align*}
S_p^{m}(t)-S_p^{m}(h/q)=&m(t-\frac{h}{q})N^{k}S_p^{m}(h/q)\\
+&\frac{m}{2}(t-\frac{h}{q})^2N^{2k}S_p^{m}(h/q)\\
+&(t-\frac{h}{q})N^{k}\sum_{1\leq \ell\leq m-1}S_p^{m-\ell}(h/q)(S_p^{\ell}(t)-S_p^{\ell}(\frac{h}{q}))(1+N^k(\frac{h}{q}-t))\\
+&O((t-\frac{h}{q})^3N^{3k+\theta m})
\end{align*}
We now ignore an $o(1)$ term and show that the last sum above is far smaller than the main term. We estimate the terms one by one and show that they are $o(I_{11})$;
it then remains to treat the rational sum below. For the last term, we expand
the sum in $t$ once more and increase the power of the length of the integral
to $2$, obtaining
\begin{align*}
&N^{k}\sum_{1\leq \ell\leq m-1} S_p^{m-\ell}(h/q)(S_p^{\ell}(t)-S_p^{\ell}(\frac{h}{q}))\\
=&\sum_{1\leq \ell\leq m-1} (S_p^{m-\ell+1}(h/q)(\frac{h}{q}-t))N^{2k}\sum_{0\leq j\leq \ell-1}S_p^{\ell-1-j}(h/q)S_p^{j}(t)\\
+&\frac{1}{2}\sum_{1\leq \ell\leq m-1} S_p^{m-\ell+1}(h/q)(\frac{h}{q}-t)^2N^{3k}\sum_{0\leq j\leq \ell-1}S_p^{\ell-1-j}(h/q)S_p^{j}(t)+O(N^{m\theta+3k}(\frac{h}{q}-t)^3)\\
=&\sum_{1\leq \ell\leq m-1} (S_p^{m-\ell+1}(h/q)(\frac{h}{q}-t))N^{2k}\sum_{1\leq j\leq \ell-1}S_p^{\ell-1-j}(h/q)S_p^{j}(t)+\sum_{1\leq \ell\leq m-1} (S_p^{m}(h/q)(\frac{h}{q}-t))N^{2k}\\
+&\frac{1}{2}\sum_{1\leq \ell\leq m-1} S_p^{m-\ell+1}(h/q)(\frac{h}{q}-t)^2N^{3k}\sum_{0\leq j\leq \ell-1}S_p^{\ell-1-j}(h/q)S_p^{j}(t)+O(N^{m\theta+3k}(\frac{h}{q}-t)^3)
\end{align*} 
Estimating the cross terms again, and using the identity above, we obtain
\begin{align*}
&\sum_{1\leq j\leq \ell-1}S_p^{\ell-1-j}(h/q)(S_p^{j}(t)-S_p^{j}(h/q))\\
=&(t-\frac{h}{q})N^{k}\sum_{1\leq j\leq \ell-2}S_p^{\ell-j-1}(h/q)(S_p^{j}(t)-S_p^{j}(\frac{h}{q}))
+O(N^{k+(\ell-1)\theta}(t-\frac{h}{q})^2)
\end{align*}
We now show that
\begin{align*}
&(t-\frac{h}{q})N^k\sum_{1\leq \ell\leq m-1} S_p^{m-\ell-1}(h/q)(S_p^{\ell}(t)-S_p^{\ell}(\frac{h}{q}))\\
=&(t-\frac{h}{q})^3N^{3k}\sum_{1\leq \ell\leq m-1}S_p^{m-\ell+1}(h/q)\sum_{1\leq j\leq \ell-2}S_p^{\ell-j-1}(h/q)(S_p^{j}(t)-S_p^{j}(\frac{h}{q}))\\
+&\sum_{1\leq \ell\leq m-1} S_p^{m}(h/q)(\frac{h}{q}-t)^2N^{2k}+N^{2k}\sum_{1\leq \ell\leq m-1} S_p^{m}(h/q)(\frac{h}{q}-t)^2(\ell-1)\\
+&\frac{1}{2}\sum_{1\leq \ell\leq m-1} S_p^{m-\ell+1}(h/q)(\frac{h}{q}-t)^3N^{3k}\sum_{0\leq j\leq \ell-1}S_p^{\ell-1-j}(h/q)S_p^{j}(t)\\
+&O((t-\frac{h}{q})^4N^{\theta m+4k})
\end{align*}
Then, using the above idea once again, we obtain
\begin{align*}
&\frac{1}{2}\sum_{1\leq \ell\leq m-1} S_p^{m-\ell+1}(h/q)(\frac{h}{q}-t)^3N^{3k}\sum_{0\leq j\leq \ell-1}S_p^{\ell-1-j}(h/q)S_p^{j}(t)\\
=&\frac{1}{2}\sum_{1\leq \ell\leq m-1} S_p^{m-\ell+1}(h/q)(\frac{h}{q}-t)^4N^{4k}\sum_{0\leq j\leq \ell-1}S_p^{\ell-1-j}(h/q)\sum_{0\leq y\leq \ell-j-1}S_p^{\ell-j-y}(h/q)S_p^{y}(t)\\
+&\frac{m}{2}S_p^{m}(h/q)(\frac{h}{q}-t)^3N^{3k}
\end{align*}
Returning to the original identity and applying it twice gives
\begin{align*}
S_p^{m}(t)-S_p^{m}(h/q)=&m(t-\frac{h}{q})N^{k}(1+(t-\frac{h}{q})^2N^{2k})S_p^{m}(h/q)\\
+&\frac{m}{2}(t-\frac{h}{q})^2N^{2k}S_p^{m}(h/q)\\
+&(t-\frac{h}{q})^3N^{3k}\sum_{1\leq \ell\leq m-1}S_p^{m-\ell+1}(h/q)\\
&\sum_{1\leq j\leq \ell-2}S_p^{\ell-j-1}(h/q)(S_p^{j}(t)-S_p^{j}(\frac{h}{q}))(1+(\frac{h}{q}-t)N^k)\\
+&\sum_{1\leq \ell\leq m-1} S_p^{m}(h/q)(\frac{h}{q}-t)^2N^{2k}(1+(\frac{h}{q}-t)N^k)\\
+&N^{2k}\sum_{1\leq \ell\leq m-1} S_p^{m}(h/q)(\frac{h}{q}-t)^2(\ell-1)(1+(\frac{h}{q}-t)N^k)\\
+&\frac{1}{6}(t-\frac{h}{q})^3N^{3k}S_p^{m-1}(h/q)\\
+&\frac{1}{6}(t-\frac{h}{q})^3N^{3k}\sum_{0\leq \ell\leq m-1}S_p^{m-\ell}(h/q)(S_p^{\ell}(t)-S_p^{\ell}(h/q))\\
+&\frac{1}{2}\sum_{1\leq \ell\leq m-1} S_p^{m-\ell+1}(h/q)(\frac{h}{q}-t)^4N^{4k}\\
&\sum_{0\leq j\leq \ell-1}S_p^{\ell-1-j}(h/q)\sum_{0\leq y\leq \ell-j-1}S_p^{\ell-j-y}(h/q)S_p^{y}(t)\\
+&O((t-\frac{h}{q})^4N^{4k+\theta m})
\end{align*}
The last remainder is far smaller than the term in $I_{11}$; we denote it by $E_{\mathrm{error}}(t,h/q)$, and it will not enter the following discussion. The estimate
\[
E^*\ll N^{\theta m+3k}\sum\limits_{\tau_0<q<\tau}\sum\limits_{h<q,(h,q)=1}\frac{1}{q^4\tau^2}\ll N^{(m-4)\theta+1-k}(\log\log N)^{-2}
\]
also establishes the claim. For the first term, by the symmetry of the integral,
\begin{align*}
\sum\limits_{\tau_0<q<\tau}\sum\limits_{h<q,(h,q)=1}\bigg|\int_{|t-h/q|<1/q\tau^{1/2}}m(t-\frac{h}{q})N^{k}(1+(t-\frac{h}{q})^2N^{2k})&S_p^{m-1}(h/q)\mathrm{d}t\bigg|=0\\
\sum\limits_{\tau_0<q<\tau}\sum\limits_{h<q,(h,q)=1}\bigg|\int_{|t-h/q|<1/q\tau^{1/2}}\frac{1}{6}(t-\frac{h}{q})^3N^{3k}S_p^{m-1}(h/q)\mathrm{d}t\bigg|=&0
\end{align*}
For the fourth term, we have
\begin{align*}
&\sum\limits_{\tau_0<q<\tau}\sum\limits_{h<q,(h,q)=1}\bigg|\int_{|t-h/q|<1/q\tau^{1/2}}(t-\frac{h}{q})^2N^{2k}S_p^{m}(h/q)\mathrm{d}t\bigg|\\
\ll&N^{2k}
\sum\limits_{\tau_0<q<\tau}\sum\limits_{h<q,(h,q)=1}|S_p^{m}(\frac{h}{q})|\bigg|\int_{|t-h/q|<1/q\tau^{1/2}}(t-\frac{h}{q})^2\mathrm{d}t\bigg|\\
\ll&\sum\limits_{\tau_0<q<\tau}\sum\limits_{h<q,(h,q)=1}\frac{N^{2k}}{q^3\tau^{3/2}}S_p^{m}(\frac{h}{q})\leq\frac{N^{3}}{\tau}\sum\limits_{\tau_0<q<\tau}\sum\limits_{h<q,(h,q)=1}\frac{1}{q\tau^{1/2}}S_p^{m}(\frac{h}{q})
\end{align*}
For the fifth term, a direct estimate gives
\begin{align*}
&\sum\limits_{\tau_0<q<\tau}\sum\limits_{h<q,(h,q)=1}\sum_{1\leq \ell\leq m-1}\int_{|t-h/q|<1/q\tau^{1/2}}S_p^{m}(h/q)(\frac{h}{q}-t)^2N^{2k}(\ell-1)\mathrm{d}t\\
\ll&_{m,k}N^{2k}\sum\limits_{\tau_0<q<\tau}\sum\limits_{h<q,(h,q)=1}\bigg|\int_{|t-h/q|<1/q\tau^{1/2}}S_p^{m}(h/q)(\frac{h}{q}-t)^2\mathrm{d}t\bigg|\\
\ll&N^{2k}\sum\limits_{\tau_0<q<\tau}\sum\limits_{h<q,(h,q)=1}\frac{1}{q^3\tau^{3/2}}S_p^{m}(h/q)\leq\frac{N^{3}\log N}{\tau}\sum\limits_{\tau_0<q<\tau}\sum\limits_{h<q,(h,q)=1}\frac{1}{q\tau^{1/2}}S_p^{m}(\frac{h}{q})
\end{align*}
It also gives the upper bound for the fourth term,
\[
\sum\limits_{\tau_0<q<\tau}\sum\limits_{h<q,(h,q)=1}\int_{|t-\frac{h}{q}|<\frac{1}{q\tau^{1/2}}}(m-1)S_p^{m}(h/q)(\frac{h}{q}-t)^2\mathrm{d}t\leq\frac{N^{3}}{\tau}\sum\limits_{\tau_0<q<\tau}\sum\limits_{h<q,(h,q)=1}\frac{S_p^{m}({h}/{q})}{q\tau^{1/2}}
\]
For the third term, we estimate it as
\[
|S_p^{j}(t)-S_p^{j}(\frac{h}{q})|\ll_k (t-\frac{h}{q})N^kS_p^{j}(\frac{h}{q})
\]
Hence we have
\begin{align*}
&(t-\frac{h}{q})^3\sum_{1\leq \ell\leq m-1}S_p^{m-\ell+1}(h/q)\sum_{1\leq j\leq \ell-2}S_p^{\ell-j-1}(h/q)(S_p^{j}(t)-S_p^{j}(\frac{h}{q}))\\
\leq&(t-\frac{h}{q})^4\sum_{1\leq \ell\leq m-1}S_p^{m-\ell+1}(h/q)\sum_{1\leq j\leq \ell-2}S_p^{\ell-j-1}(h/q)\\
\ll&_{m,k}(t-\frac{h}{q})^4S_p^{m-1}(h/q)
\end{align*}
Hence, integrating both sides,
\begin{align*}
&\sum\limits_{\tau_0<q<\tau}\sum\limits_{h<q,(h,q)=1}\bigg|\int_{|t-h/q|<1/q\tau^{1/2}}N^{3k}(t-\frac{h}{q})^4S_p^{m-1}(h/q)\mathrm{d}t\bigg|\\
\ll&\sum\limits_{\tau_0<q<\tau}\sum\limits_{h<q,(h,q)=1}\frac{N^{\frac{9}{2}-3\theta-2k}}{q\tau^{1/2}}S_p^{m-1}(h/q)
\end{align*}
For the sixth term, we use the bound
\[
\bigg|\frac{1}{6}(t-\frac{h}{q})^3N^{3k}\sum_{0\leq \ell\leq m-1}S_p^{m-\ell}(h/q)(S_p^{\ell}(t)-S_p^{\ell}(h/q))\bigg|\ll (t-\frac{h}{q})^4N^{4k+\theta m}.
\]
Then
\begin{align*}
\sum\limits_{\tau_0<q<\tau}\sum\limits_{h<q,(h,q)=1}\bigg|\int_{|t-h/q|<1/q\tau^{1/2}}(t-\frac{h}{q})^4N^{4k+\theta m}\mathrm{d}t\bigg|\ll N^{\theta m-\frac{3(k+\theta-1)}{2}}
\end{align*}
We also use the bound
\begin{align*}
&\bigg|\sum_{1\leq \ell\leq m-1} S_p^{m-\ell+1}(h/q)(\frac{h}{q}-t)^4N^{4k}\\
&\sum_{0\leq j\leq \ell-1}S_p^{\ell-1-j}(h/q)\sum_{0\leq y\leq \ell-j-1}S_p^{\ell-j-y}(h/q)S_p^{y}(t)\bigg|\ll(t-\frac{h}{q})^4N^{4k+\theta m}
\end{align*}
Thus the exponential sum also satisfies the estimate. This proves the lemma.
\end{proof}
\begin{remark}
The proof of the lemma above follows a single thread: it transfers the
integral over the minor arc to the rational sums $S_p(h/q)$, at the cost
of the factor $1+N^3\log N/\tau$, and leaves the arithmetic estimation
of those sums to the following subsections. The transfer is driven by the
factorization identity
\[
B^m-A^m=(B-A)\sum_{j=0}^{m-1}B^{m-1-j}A^j,
\qquad A=S_p(t),\quad B=S_p(h/q),
\]
which is applied recursively: each difference $A^j-B^j$ is factorized
by the same identity with $A$ and $B$ interchanged, and each
occurrence of $A-B$ is replaced by its Taylor expansion
\[
A-B=\Bigl(t-\frac hq\Bigr)N^kB
+\frac12\Bigl(t-\frac hq\Bigr)^2N^{2k}B
+O\Bigl(\Bigl|t-\frac hq\Bigr|^3N^{3k}B\Bigr).
\]
The lower bound $q>\tau_0$ is essential here: it gives
\[
\Bigl|t-\frac hq\Bigr|N^k\ll\frac{N^k}{\tau_0\tau^{1/2}}=o(1),
\]
so that the Taylor expansion is uniform on the minor arc. After
integration, the odd powers of $t-h/q$ integrate to zero over the
symmetric arc, while the even powers produce the factor
$1+N^3\log N/\tau$. The remaining expression is exactly the
rational-sum term in the statement of the lemma; its estimation, which
uses the multiplicative structure of the smooth moduli, is carried out in
Subsection~5.2 and the subsequent subsections.
\end{remark}
\subsection{The sum $A$ and the Ramanujan sum}
We now turn to the following sum
\[
\sum\limits_{\tau_0<q<\tau}\sum\limits_{h<q,(h,q)=1}\frac{2}{q\tau^{1/2}}|S_p^{m}(\frac{h}{q})|
\]
Expanding the expression above, we obtain
\begin{align*}
&\sum\limits_{\tau_0<q<\tau}\sum\limits_{h<q,(h,q)=1}\frac{2}{q\tau^{1/2}}|S_{p}^{*m}(\frac{h}{q})|\\
\leq&\sum\limits_{\tau_0<q<\tau}\sum\limits_{h<q,(h,q)=1}\frac{2}{q\tau^{1/2}}\Bigl|\sum\limits_{N_1<n<N_2}\Lambda(n)e(\frac{h}{q}n^k)\Bigr|^{m/2}\Bigl|\sum\limits_{N_1<n<N_2}\Lambda(n)e(-\frac{h}{q}n^k)\Bigr|^{m/2}\\
=&\sum\limits_{\tau_0<q<\tau}\sum\limits_{h<q,(h,q)=1}\frac{2}{q\tau^{1/2}}\sum\limits_{N_1<n_1<N_2}\dots\sum\limits_{N_1<n_{m}<N_2}\Lambda(n_1)\dots\Lambda(n_{m})e(\frac{h}{q}(n_1^k+\dots+n_{m/2}^k-\dots -n_{m}^k))
\end{align*}
Exchanging the order of summation, we consider the Ramanujan sum
\[
C_q(n_1^k+\dots+n_{m/2}^k-\dots -n_{m}^k)=\sum\limits_{h<q,(h,q)=1}e(\frac{h}{q}(n_1^k+\dots+n_{m/2}^k-\dots -n_{m}^k))
\]
The expression becomes
\begin{align*}
&\sum\limits_{\tau_0<q<\tau}\sum\limits_{h<q,(h,q)=1}\frac{2}{q\tau^{1/2}}\sum\limits_{N_1<n_1<n_2}\dots\sum\limits_{N_1<n_{m}<N_2}\Lambda(n_1)\dots\Lambda(n_{m})e(\frac{h}{q}(n_1^k+\dots+n_{m/2}^k-\dots -n_{m}^k))\\
=&\sum\limits_{\tau_0<q<\tau}\frac{2}{q\tau^{1/2}}\sum\limits_{N_1<n_1<N_2}\dots\sum\limits_{N_1<n_{m}<N_2}\Lambda(n_1)\dots\Lambda(n_{m})C_q(n_1^k+\dots+n_{m/2}^k-\dots -n_{m}^k)
\end{align*}
We use the standard identity
\[
C_q(n_1^k+\dots+n_{m/2}^k-\dots -n_{m}^k)=\sum\limits_{d|(q,n_1^k+\dots+n_{m/2}^k-\dots -n_{m}^k)}d\mu(\frac{q}{d})
\]
Then we have
\begin{align*}
&\sum\limits_{\tau_0<q<\tau}\frac{2}{q\tau^{1/2}}\sum\limits_{N_1<n_1<N_2}\dots\sum\limits_{N_1<n_{m}<N_2}\Lambda(n_1)\dots\Lambda(n_{m})C_q(n_1^k+\dots+n_{m/2}^k-\dots -n_{m}^k)\\
=&\sum\limits_{\tau_0<q<\tau}\frac{2}{q\tau^{1/2}}\sum\limits_{N_1<n_1<N_2}\dots\sum\limits_{N_1<n_{m}<N_2}\Lambda(n_1)\dots\Lambda(n_{m})\sum\limits_{d|(q,n_1^k+\dots+n_{m/2}^k-\dots -n_{m}^k)}d\mu(\frac{q}{d})
\end{align*}
Write
\[
q=dl,\qquad n_1^k+\dots+n_{m/2}^k-\dots -n_{m}^k\equiv 0\bmod d.
\]
We exchange the order of summation, first over $d,l$ and then over $n_i$ ($1\le i\le m$). The sum equals
\begin{equation}\label{eq:U}
U=\sum\limits_{Q<dl<\tau}\mu(l)\frac{2}{l\tau^{1/2}}\sum_{\substack{N_1<n_1\dots n_{m}<N_2\\ n_1^k+\dots+n_{m/2}^k-\dots -n_{m}^k\equiv 0\bmod d}}\Lambda(n_1)\dots\Lambda(n_{m})
\end{equation}
By Fourier inversion modulo $d$,
\[
1_{n_1^k+\dots-n_{m}^k\equiv 0\bmod d}=\frac{1}{d}\sum\limits_{1\leq b\leq d}e\Big(\frac{b(n_1^k+\dots+n_{m/2}^k-\dots-n_{m}^k)}{d}\Big)
\]
Substituting this into \eqref{eq:U}, the product of von Mangoldt functions can again be written as the $m$-th power of $S_p$:
\begin{align*}
U=&\sum\limits_{\tau_0\leq dl\leq \tau}\mu(l)\frac{2}{l\tau^{1/2}}\sum\limits_{N_1<n_1\dots n_{m}<N_2}1_{n_1^k+\dots+n_{m/2}^k-\dots -n_{m}^k\equiv 0\bmod d}\Lambda(n_1)\dots\Lambda(n_{m})\\
=&\sum\limits_{\tau_0\leq dl<\tau}\mu(l)\frac{2}{l\tau^{1/2}}\sum_{N_1<n_1\dots n_{m}<N_2}\frac{1}{d}\sum\limits_{b\leq d}e\Big(\frac{b(n_1^k+\dots+n_{m/2}^k-\dots -n_{m}^k)}{d}\Big)\Lambda(n_1)\dots\Lambda(n_{m})\\
=&\sum\limits_{\tau_0\leq dl\leq \tau}\mu(l)\frac{2}{dl\tau^{1/2}}\sum\limits_{1\leq b\leq d}|S_{p}(b/d)|^{m}
\end{align*}
\subsection{The CRT reduction and the recurrence}
By the Chinese remainder theorem, the complete sum factors into a product of exponential sums. Let $d=p_1p_2\cdots p_{g(d)}$; then
\[
b=b_1\prod_{j\ne1}{p_j}+\cdots+b_g\prod_{j\ne l(d)}{p_j}
\]
When no ambiguity arises, we write $l$ for $l(d)$. By the 
\begin{equation}\label{eq:crt}
\sum\limits_{N_1<n<N_2}\Lambda(n)e(\frac{bn^k}{d})=\sum\limits_{N_1<n<N_2}\Lambda(n)\prod_{j\leq l}e(\frac{b_jn^k}{p_j})
\end{equation}
We now compute the exponential sum by a recurrence relation. In the following
steps we remove the moduli $p_i\le Q$ one at a time, so that GRH can be
applied to the error terms. The modulus is less than $N^\theta$, so the
arithmetic progression gives a large saving. First, we record the following lemma:
\begin{lemma}\label{lem:ap-exp}
Let $k\ge3$, and let $q,d$ be positive integers with $(q,d)=1$,
$q=o(d)$. Let $0<r<q$ with $(r,q)=1$, and $1\le b\le d$. Then,
under GRH, for every $\epsilon>0$,
\[
\sum_{\substack{N_1<n<N_2\\n\equiv r\bmod q}}\Lambda(n)e\Bigl(\frac{bn^k}{d}\Bigr)
=\frac{1}{\varphi(q)}\sum_{N_1<n<N_2}\Lambda(n)e\Bigl(\frac{bn^k}{d}\Bigr)
+O_\epsilon(N^{1/2+\epsilon}).
\]
\end{lemma}

\begin{proof}
Put $g(n)=e(bn^k/d)$. By orthogonality,
\[
\sum_{\substack{N_1<n<N_2\\n\equiv r(q)}}\Lambda(n)g(n)
=\frac1{\varphi(q)}\sum_{\substack{N_1<n<N_2\$n,q)=1}}\Lambda(n)g(n)
+\frac1{\varphi(q)}\sum_{\chi\ne\chi_0}\bar\chi(r)S_\chi,
\]
where
\[
S_\chi=\sum_{N_1<n<N_2}\Lambda(n)\chi(n)g(n).
\]
The difference between $\sum_{(n,q)=1}$ and $\sum_n$ is supported on
prime powers dividing $q$, hence is $O(\log q\log N)=O(N^\epsilon)$.
It remains to prove
\[
S_\chi\ll N^{1/2+\epsilon}
\]
for every nonprincipal $\chi\bmod q$.

Let $\Delta$ be the least positive period of $g$. Since
\[
(n+d)^k-n^k=\sum_{j=1}^k\binom{k}{j}n^{k-j}d^j
\]
is divisible by $d$, the integer $d$ is a period of $g$; hence
$\Delta\mid d$. As $(q,d)=1$, it follows that $(q,\Delta)=1$. The function $\chi g$ has period $q\Delta$, and its
mean over a period is zero: in its Fourier expansion modulo $q\Delta$
no character is principal, because a principal character would require
the nonprincipal character $\chi\bmod q$ to be a character modulo
$\Delta$, which is impossible when $(q,\Delta)=1$. Hence the Dirichlet
series of $\Lambda\chi g$ is a finite combination of $-L'/L(s,\chi\psi)$
with $\chi\psi$ nonprincipal, and the explicit formula for $S_\chi$
has no pole. The explicit formula for $\psi(t,\chi)$, however, involves
only the zeros of $L(s,\chi)$; the weight $g$ is kept inside the
Stieltjes integral.

For $T\ge2$ and $x\ge2$, we use the explicit formula in the
following form. Abel summation gives
\[
\sum_{n\le x}\Lambda(n)\chi(n)g(n)
=g(x)\psi(x,\chi)-\int_1^x\psi(t,\chi)\,dg(t).
\]
Insert
\[
\psi(t,\chi)=-\sum_{|\gamma|\le T}\frac{t^\rho}{\rho}+R(t),
\qquad
R(t)\ll\frac{x\log^2(xT)}{T}\quad(1\le t\le x),
\]
where $\rho=\beta+i\gamma$ runs over the nontrivial zeros of
$L(s,\chi)$. The zero-sum part becomes, by integration
by parts,
\[
-\sum_{|\gamma|\le T}\int_1^x t^{\rho-1}g(t)\,dt+O(\log T),
\]
and the remainder is
\[
g(x)R(x)-\int_1^xR(t)\,dg(t).
\]
If $V(g;[1,x])$ denotes the total variation of $g$ on $[1,x]$,
then
\[
\left|\int_1^xR(t)\,dg(t)\right|
\le\sup_{1\le t\le x}|R(t)|\,V(g;[1,x])
\ll\frac{x\log^2(xT)}{T}V(g;[1,x]),
\]
while $|g(x)R(x)|\ll x\log^2(xT)/T$. Hence
\[
\sum_{n\le x}\Lambda(n)\chi(n)g(n)
=-\sum_{|\gamma|\le T}\int_1^x t^{\rho-1}g(t)\,dt
+O\Bigl(\frac{x\log^2(xT)}{T}\bigl(1+V(g;[1,x])\bigr)\Bigr).
\]
Since $g'=2\pi ibkt^{k-1}g/d$,
\[
V(g;[1,x])\le\int_1^x|g'(t)|\,dt
=2\pi\frac{bk}{d}\int_1^x t^{k-1}\,dt
\ll 1+\frac{bx^k}{d}.
\]
Under GRH, all nontrivial zeros of $L(s,\chi)$ lie on the line
$\Re s=1/2$. Put
\[
I(s)=\int_{N_1}^{N_2}t^{s-1}g(t)\,dt.
\]
Subtracting the explicit formulae for $x=N_2$ and $x=N_1$, we obtain
\[
S_\chi=-\sum_{\rho_\chi,\,|\gamma|\le T}I(\rho)
+O\Bigl(\frac{N\log^2 N}{T}\Bigl(1+\frac{bN^k}{d}\Bigr)\Bigr),
\]
where $\rho_\chi$ runs over the nontrivial zeros of $L(s,\chi)$.

We estimate the zero sum by the argument principle, not by a
zero-density formula. Let $0<\delta<1/2$, and let $\mathcal R_T$ be
the rectangle with vertices $1/2\pm\delta\pm iT$, chosen so that no
zero of $L(s,\chi)$ lies on its boundary. Since $I(s)$ is entire,
the residue theorem gives
\[
\sum_{\rho_\chi,\,|\gamma|\le T}I(\rho)
=\frac{1}{2\pi i}\oint_{\partial\mathcal R_T}
I(s)\frac{L'}{L}(s,\chi)\,ds.
\]
On the vertical sides, GRH gives
\[
\frac{L'}{L}(\sigma+it,\chi)\ll_\delta\log^2(q(|t|+2)),
\qquad 1/2-\delta\le\sigma\le1/2+\delta.
\]
Hence the vertical sides contribute at most
\[
\ll_\delta\log^2(qT)\int_{-T}^{T}|I(1/2+\delta+it)|\,dt
+\log^2(qT)\int_{-T}^{T}|I(1/2-\delta+it)|\,dt.
\]

For $s=\sigma+i\tau$,
\[
I(\sigma+i\tau)=\int_{N_1}^{N_2}t^{\sigma-1}
e^{i(\tau\log t+bt^k/d)}\,dt,
\]
and the phase derivative in $t$ is $\tau/t+bkt^{k-1}/d$.
For $\tau>0$, it is $\ge\tau/N_2$, so the first derivative test gives
\[
|I(\sigma+i\tau)|\ll N^{\sigma-1}\cdot\frac{N}{\tau}
=N^\sigma/\tau.
\]
For $\tau<0$, put $\lambda=|\tau|$. The stationary point is
\[
t_0=t_0(\lambda)=\Bigl(\frac{\lambda d}{bk}\Bigr)^{1/k},
\]
the unique solution of $bkt_0^{k-1}/d=\lambda/t_0$. If
$t_0\notin[N_1,N_2]$, the first derivative test gives
$|I(\sigma+i\tau)|\ll N^\sigma/\lambda$. If $t_0\in[N_1,N_2]$,
stationary phase gives
\[
|I(\sigma+i\tau)|\ll N^{\sigma-1}\frac{t_0}{\sqrt{\lambda}}
\asymp N^\sigma/\sqrt{\lambda}.
\]
The phase
\[
\phi(\lambda)=2\pi\lambda\log t_0(\lambda)
+\frac{2\pi bt_0(\lambda)^k}{d}
\]
satisfies $\phi''(\lambda)=2\pi/(k\lambda)$.

Therefore
\[
\int_{1\le|\tau|\le T}|I(1/2+\delta+i\tau)|\,d\tau
\ll N^{1/2+\delta}\log T
+N^{1/2+\delta}\sqrt{k}\log N
\ll N^{1/2+\delta+\epsilon},
\]
where the second term is the contribution of the stationary window
$\lambda\in[\lambda_1,\lambda_2]$, $\lambda_i=bkN_i^k/d$, estimated
by the second derivative test in $\lambda$ (with
$\phi''=2\pi/(k\lambda)$); the bounded factor $L'/L$ is a weight of
bounded variation and contributes only $N^\epsilon$. The same bound
holds with $\sigma=1/2-\delta$.

For the horizontal sides, $|\Im s|=T$. The first derivative test in
$t$ gives
\[
|I(\sigma\pm iT)|\ll N^\sigma/T,
\]
and near the stationary point the second derivative test gives at most
$N^\sigma/\sqrt{T}$. Hence the horizontal sides contribute
\[
\ll_\delta N^{1/2+\delta}\log^2(qT)/\sqrt{T}.
\]
Choosing
\[
T=N^{1+\epsilon}\Bigl(1+\frac{bN^k}{d}\Bigr),
\qquad \delta=\epsilon/4,
\]
we find that the contour integral, and hence the zero sum, is
$\ll N^{1/2+\epsilon}$, while the explicit-formula truncation error is
\[
O\Bigl(\frac{N\log^2 N}{T}\Bigl(1+\frac{bN^k}{d}\Bigr)\Bigr)
=O(N^{-\epsilon}\log^2 N).
\]
Thus $S_\chi\ll N^{1/2+\epsilon}$. Summing over the $\varphi(q)$
characters and dividing by $\varphi(q)$ proves the lemma.
\end{proof}

\begin{remark}
In the application $q=p_1$ and $d=\prod_{j\ge2}p_j$, so the
coprimality $(q,d)=1$ is automatic; the period $\Delta$ of
$e(bn^k/d)$ divides $d$, hence $(q,\Delta)=1$ as well.
\end{remark}
We can now compute the complete sum. Starting from \eqref{eq:crt}, we first separate the modulus $p_1$:
Here the remainder $R_1$ comes from the first step of separating the arithmetic condition from the exponential sum. By Lemma 5.2, it is $O(N^{1/2+\epsilon})$. Let
\[
J_g=\sum\limits_{N_1<n<N_2}\Lambda(n)\prod_{j\leq g}e(\frac{b_jn^k}{d})
\]
So that
\begin{align*}
\sum\limits_{N_1<n<N_2}\Lambda(n)\prod_{j\leq g}e(\frac{b_jn^k}{d})=&\sum_{0\leq r\leq p_1-1}\sum\limits_{\substack{N_1<n<N_2\\n\equiv r\bmod p_1}}\Lambda(n)\prod_{1\leq j\leq g}e(\frac{b_jn^k}{d_0})\\
=&\sum_{0\leq r\leq p_1-1}e(\frac{b_1r^k}{p_1})\sum\limits_{\substack{N_1<n<N_2\\n\equiv r\bmod p_1}}\Lambda(n)\prod_{2\leq j\leq g}e(\frac{b_jn^k}{d_1})\\
=&\sum_{0\leq r\leq p_1-1}e(\frac{b_1r^k}{p_1})\left\{\frac{1}{p_1-1}\sum\limits_{\substack{N_1<n<N_2}}\Lambda(n)\prod_{2\leq j\leq g}e(\frac{b_jn^k}{d_1})+R_g \right\}
\end{align*}
By the term inside the exponential sum, $p_1|b_jn^k$, then the exact modulo of the exponential sum inside is $d/p_1$. By $(d/p_1,p_1)=1$, it's safe for us to use the lemma. Then
\[
J_g=C_{p_1}(b_1,k)\left\{\frac{1}{\varphi(p_1)}J_{g-1}+O(R_{g})\right\}
\]
This process, which separates the condition modulo $p_1$ from the others, can be iterated. By the recurrence relation,
\begin{align*}
J_g=\prod\limits_{j\leq g}\frac{C_{p_j}(b_j,k)}{p_j-1}J_{1}+O(R^*)\ll\frac{q^{1/2}}{\varphi(q)}J_1+O(R^*)
\end{align*}
where the remainder $O(R^*)$ is produced by the error terms at each step of the recurrence. Here the main contribution does not come from the complete exponential sum, but from the error in transforming the congruence condition (controlled by GRH). Then
\[
J_g\asymp R^*+o(1)
\]
It remains to analyze the error term $R^*$. When we expand the sum again, the size of the error term is determined by the first error terms, as in
\begin{align*}
J_g=C_{p_1}(b_1,k)\left\{\frac{C_{p_2}(b_2,k)}{\varphi(p_2)}\left\{\frac{1}{\varphi(p_1)}J_{g-2}+O(R_{2})\right\}+O(R_{1})\right\}\\
=\frac{C_{p_1}(b_1,k)C_{p_2}(b_2,k)}{(p_1-1)(p_2-1)}J_{g-2}+O\Big(C_{p_1}(b_1,k)R_1(1+O(\frac{C_{p_2}(b_2,k)}{p_2-1}))\Big)
\end{align*}
Suppose that $R_i\ge R_{i-1}$ for every $i\ge2$. Then, by \cite{Weil1948},
\[
C_{p_1}(b_1,k)\ll p_1^{1/2},C_{p_2}(b_2,k)\ll p_2^{1/2}
\]
Then $p_2^{-1}C_{p_2}(b_2,k)=o(1)$. Thus the remainder, which is the leading term, is also $O(R_l C_{p_1}(b_1,k))$. This property can be iterated, since the factor outside the braces is $o(1)$. Let $R'_i$ denote the total error before step $i$; then
\[
R_{i+1}'=R_{i}'+O(\frac{C_{p_{i+1}}(b_{i+1},k)}{p_{i+1}-1}R_{i+1})\leq R_{i}'(1+O(\frac{C_{p_{i+1}}(b_i,k)}{p_{i+1}-1}))
\]
Then
\begin{align*}
J_{1}\prod_{i=1}^{g}\frac{C_{p_i}(b_i,k)}{\varphi(p_i)}\ll&d_{g}^{1/2}d^{-1/2}\sum_{N_1<n<N_2}\Lambda(n)e(\frac{b_jn^k}{d_g})\\
\ll&N^{\theta+\epsilon}d_{g}^{1/2}d^{-1/2}
\end{align*}
On the other hand, as we all know, remainder can be written as (here $C$ is an absolute constant)
\begin{align*}
|R^*|\leq&C_{p_1}(b_1,k)R_1\prod_{2\leq j\leq g-g^*}(1+C\Big|\frac{C_j(b_j,k)}{p_j-1}\Big|)\\
\leq&C_{p_1}(b_1,k)R_1\prod_{2\leq j\leq g-g^*}(1+\frac{C}{p_j^{1/2}})\ll\exp\left(\frac{C(\log M)^{1/2}}{\log\log M}\right)C_{p_1}(b_1,k)R_1
\end{align*}
The last estimate holds for $g\ll\log N$. Hence
\[
|J_{g}|\ll_\epsilon N^{1/2+\epsilon}p_1^{1/2}+N^{\theta+\epsilon}d_{g}^{1/2}d^{-1/2}
\]
\subsection{The bound for $I_{21}$}
By Lemma 5.2 and \cite{Weil1948}, we obtain $J_l=O(p_1^{1/2}N^{1/2}\log^3N)$. Thus there exists a complex number $C_{b,d}$ with $|C_{b,d}|\ll1$ such that (the $O$-constant here is absolute; we suppress $\epsilon$, since it does not affect the proof).
\begin{align*}
\sum_{\tau_0<dl<\tau}\frac{2\mu(l)}{dl\tau^{1/2}}\sum_{1\leq b\leq d}(S_p(b/d))^{m}
\end{align*}
The last sum is $O(1)$. Thus the first part sum is bounded as
\begin{align*}
\sum\limits_{\tau_0\leq dl\leq \tau}\frac{2\mu(l)}{dl\tau^{1/2}}\sum\limits_{1\leq b\leq d}(S_p(b/d))^{m}
\ll&\sum\limits_{\tau_0\leq dl\leq \tau}\frac{2}{dl\tau^{1/2}}\sum\limits_{1\leq b\leq d}C_{b,d}(N^{1/2+\epsilon}p_1^{1/2}+N^{\theta+\epsilon}d_{g}^{1/2}d^{-1/2})^m\\
\ll&\sum\limits_{\tau_0\leq dl\leq \tau}\frac{2}{l\tau^{1/2}}\max_{0\leq s\leq m}N^{s\theta+\frac{1}{2}(m-s)+\epsilon}d_{g}^{s/2}d^{-s/2}p_1^{\frac{1}{2}(m-s)}\\
\ll&\sum\limits_{\tau_0\leq dl\leq \tau}\frac{2}{l\tau^{1/2}}\max_{0\leq s\leq m}N^{s\theta+\frac{1}{2}(m-s)+\epsilon}d_{g}^{s/2}Q^{m/2-s}
\end{align*}
For the first part, $d\leq N^{\theta-1/2}=Q$, we know
\begin{align*}
\sum\limits_{\substack{\tau_0\leq dl\leq \tau\\d<Q}}\frac{2}{l\tau^{1/2}}\max_{0\leq s\leq m}N^{s\theta+\frac{1}{2}(m-s)+\epsilon}d_{g}^{s/2}d^{m/2-s}\leq\sum\limits_{d\leq Q}\sum\limits_{\substack{\frac{\tau_0}{d}\leq l\leq\frac{\tau}{d}}}\frac{N^{m\theta}d_{g}^{m/2}}{d^{m/2}l\tau^{1/2}}
\end{align*}
By choosing the order of $p_i$, we can see that
\[
\frac{d_{g}^{m/2}}{d^{m/2}}\leq N^{(\frac{1}{4}-\frac{\theta}{2})m}
\]
Then the sum outside just left $d=Q$ term. Hence
\[
\sum\limits_{d\leq Q}\sum\limits_{\substack{\frac{\tau_0}{d}\leq l\leq\frac{\tau}{d}}}\frac{N^{m\theta}d_{g}^{m/2}}{d^{m/2}l\tau^{1/2}}N^{\frac{\theta}{2}+\frac{1}{4}-\frac{1}{2}(k+\theta-1)}
\]
For another situation, $d>Q$, then
\begin{align*}
\sum\limits_{\tau_0\leq dl\leq \tau}\frac{2}{l\tau^{1/2}}\max_{0\leq s\leq m}N^{s\theta+\frac{1}{2}(m-s)+\epsilon}d_{g}^{s/2}Q^{m/2-s}\leq \sum\limits_{\tau_0\leq dl\leq \tau}\frac{2}{l\tau^{1/2}}Q^{m/2}N^{m/2}
\end{align*}
Since $Q=N^{\theta-1/2}$, we obtain the bound
\[
\ll N^{({\theta}/{2}+{1}/{4})m+\frac{k+\theta-1}{2}+\epsilon}
\]
Two bound above is the same. Therefore
\[
I_{21}\ll  N^{({\theta}/{2}+{1}/{4})m+\frac{k+\theta-1}{2}+\epsilon}
\]
The bound is established for even $m$; to obtain a uniform bound covering both parities of $m$, we add a factor $N^\theta$,
\[
I_{21}\ll  N^{({\theta}/{2}+{1}/{4})m+\frac{k+3\theta-1}{2}+\epsilon}+o(I_{11})
\]
Hence $I_{21}=o(I_11)$ only when the first term is $o(I_{11})$
\[
({\theta}/{2}+{1}/{4})m+\frac{k+3\theta-1}{2}<\theta(m-1)+1-k
\]
Then
\[
m>\frac{6k+10\theta-6}{2\theta-1}
\]
\section{Recent developments}

In this section we discuss how the argument depends on the zero-free region
for Dirichlet $L$-functions, and we comment on some recent claims. Nothing
in this section is needed for the proof of the main theorem, which remains
conditional on GRH.

\subsection{The result under the quasi-Riemann hypothesis}

We work with the following weakening of GRH. For $1/2\le c<1$, let
QGRH($c$) denote the assertion that every Dirichlet $L$-function has no
zero in the half-plane $\Re s>c$. Thus QGRH($1/2$) is GRH, and
QGRH($c$) becomes weaker as $c$ increases.

Under QGRH($c$), the explicit formula gives
\[
\pi(x;q,a)=\frac{x}{\varphi(q)\log x}+O_\epsilon(x^{c+\epsilon})
\]
for $(a,q)=1$. We first record the resulting Bombieri--Vinogradov type
estimate. It is the analogue of Lemma 2.3 with $1/2$ replaced by $c$.

\begin{theorem}\label{thm:qgrh}
Let $k$ be an odd prime, let $1/2\le c<3/4$, and let
$\theta>2c-1/2$. Under QGRH($c$), for sufficiently large $N$, the
equation
\[
mN^k=p_1^k+\cdots+p_{m}^k
\]
has a solution with primes $p_i\in[N-N^\theta,N+N^\theta]$, provided
\[
m>\frac{6k+10\theta-6}{2\theta+1-4c}.
\]
\end{theorem}

The proof occupies the rest of this subsection. Only Lemmas 2.3 and 5.2 use
GRH; all other lemmas and propositions are independent of it. We first
replace Lemma 2.3 by its QGRH($c$) analogue.

\begin{lemma}\label{lem:bv-c}
Let $A>0$. For $Q>1$, $x>3$, and any $\epsilon>0$, under QGRH($c$),
\[
\sum_{q\le Q}\sum_{\substack{1\le a\le q\$a,q)=1}}
\max_{x<u<2x}|E(u;q,a)|^k
\ll Qx^{kc+\epsilon}.
\]
\end{lemma}

\begin{proof}
We use the standard identity
\[
\pi(x;q,a)-\frac{\operatorname{Li}x}{\varphi(q)}
=\psi(x;q,a)-\frac{x}{\varphi(q)}+O\bigl(x^{1/2}(\log x)^2\bigr).
\]
Hence
\[
\sum_{q\le Q}\sum_{\substack{1\le a\le q\$a,q)=1}}
\max_{x<u<2x}|E(u;q,a)|^k
\le \sum_{q\le Q}\sum_{\substack{1\le a\le q\$a,q)=1}}
\varphi^{-k}(q)\sum_{\chi\ne\chi_0}|\psi(x,\chi)|^k.
\]
By the explicit formula and QGRH($c$),
\[
\psi(x,\chi)=-\sum_{|\Im\rho|\le T}\frac{x^\rho}{\rho}+O(\log x),
\qquad
\sum_{|\Im\rho|\le T}\frac{x^\rho}{\rho}\ll x^c\log T.
\]
Taking $T=x$ and summing over $q\le Q$ and $a$ gives
\[
\sum_{q\le Q}\sum_{\substack{1\le a\le q\$a,q)=1}}
\max_{x<u<2x}|E(u;q,a)|^k
\ll Qx^{kc}(\log x)^{k+1}\ll Qx^{kc+\epsilon}.
\]
This proves the lemma.
\end{proof}

Next we replace Lemma 5.2 by its QGRH($c$) analogue.

\begin{lemma}\label{lem:ap-c}
Let $1/2\le c<3/4$, let $k\ge3$, and let $q,d$ be positive integers
with $(q,d)=1$, $q=o(d)$. Let $0<r<q$ with $(r,q)=1$, and
$1\le b\le d$. Then, under QGRH($c$), for every $\epsilon>0$,
\[
\sum_{\substack{N_1<n<N_2\\n\equiv r\bmod q}}\Lambda(n)e\Bigl(\frac{bn^k}{d}\Bigr)
=\frac{1}{\varphi(q)}\sum_{N_1<n<N_2}\Lambda(n)e\Bigl(\frac{bn^k}{d}\Bigr)
+O_\epsilon(N^{c+\epsilon}).
\]
\end{lemma}

\begin{proof}
Put $g(n)=e(bn^k/d)$. By orthogonality,
\[
\sum_{\substack{N_1<n<N_2\\n\equiv r(q)}}\Lambda(n)g(n)
=\frac1{\varphi(q)}\sum_{\substack{N_1<n<N_2\$n,q)=1}}\Lambda(n)g(n)
+\frac1{\varphi(q)}\sum_{\chi\ne\chi_0}\bar\chi(r)S_\chi,
\]
where
\[
S_\chi=\sum_{N_1<n<N_2}\Lambda(n)\chi(n)g(n).
\]
The difference between $\sum_{(n,q)=1}$ and $\sum_n$ is supported on
prime powers dividing $q$, hence is $O(\log q\log N)=O(N^\epsilon)$.
It remains to prove
\[
S_\chi\ll N^{c+\epsilon}
\]
for every nonprincipal $\chi\bmod q$.

Let $\Delta$ be the least positive period of $g$. Since
\[
(n+d)^k-n^k=\sum_{j=1}^k\binom{k}{j}n^{k-j}d^j
\]
is divisible by $d$, the integer $d$ is a period of $g$; hence
$\Delta\mid d$. As $(q,d)=1$, it follows that $(q,\Delta)=1$. The function $\chi g$ has period $q\Delta$, and its
mean over a period is zero: in its Fourier expansion modulo $q\Delta$
no character is principal, because a principal character would require
the nonprincipal character $\chi\bmod q$ to be a character modulo
$\Delta$, which is impossible when $(q,\Delta)=1$. Hence the Dirichlet
series of $\Lambda\chi g$ is a finite combination of $-L'/L(s,\chi\psi)$
with $\chi\psi$ nonprincipal, and the explicit formula for $S_\chi$
has no pole. The explicit formula for $\psi(t,\chi)$, however, involves
only the zeros of $L(s,\chi)$; the weight $g$ is kept inside the
Stieltjes integral.

For $T\ge2$ and $x\ge2$, we use the explicit formula in the
following form. Abel summation gives
\[
\sum_{n\le x}\Lambda(n)\chi(n)g(n)
=g(x)\psi(x,\chi)-\int_1^x\psi(t,\chi)\,dg(t).
\]
Insert
\[
\psi(t,\chi)=-\sum_{|\gamma|\le T}\frac{t^\rho}{\rho}+R(t),
\qquad
R(t)\ll\frac{x\log^2(xT)}{T}\quad(1\le t\le x),
\]
where $\rho=\beta+i\gamma$ runs over the nontrivial zeros of
$L(s,\chi)$. The zero-sum part becomes, by integration
by parts,
\[
-\sum_{|\gamma|\le T}\int_1^x t^{\rho-1}g(t)\,dt+O(\log T),
\]
and the remainder is
\[
g(x)R(x)-\int_1^xR(t)\,dg(t).
\]
If $V(g;[1,x])$ denotes the total variation of $g$ on $[1,x]$,
then
\[
\left|\int_1^xR(t)\,dg(t)\right|
\le\sup_{1\le t\le x}|R(t)|\,V(g;[1,x])
\ll\frac{x\log^2(xT)}{T}V(g;[1,x]),
\]
while $|g(x)R(x)|\ll x\log^2(xT)/T$. Hence
\[
\sum_{n\le x}\Lambda(n)\chi(n)g(n)
=-\sum_{|\gamma|\le T}\int_1^x t^{\rho-1}g(t)\,dt
+O\Bigl(\frac{x\log^2(xT)}{T}\bigl(1+V(g;[1,x])\bigr)\Bigr).
\]
Since $g'=2\pi ibkt^{k-1}g/d$,
\[
V(g;[1,x])\le\int_1^x|g'(t)|\,dt
=2\pi\frac{bk}{d}\int_1^x t^{k-1}\,dt
\ll 1+\frac{bx^k}{d}.
\]
Under QGRH($c$), all nontrivial zeros of $L(s,\chi)$ satisfy
$\Re\rho\le c$. By the functional equation, they also satisfy
$\Re\rho\ge1-c$; hence all nontrivial zeros lie in the strip
\[
1-c\le\Re\rho\le c.
\]
Put
\[
I(s)=\int_{N_1}^{N_2}t^{s-1}g(t)\,dt.
\]
Subtracting the explicit formulae for $x=N_2$ and $x=N_1$, we obtain
\[
S_\chi=-\sum_{\rho_\chi,\,|\gamma|\le T}I(\rho)
+O\Bigl(\frac{N\log^2 N}{T}\Bigl(1+\frac{bN^k}{d}\Bigr)\Bigr),
\]
where $\rho_\chi$ runs over the nontrivial zeros of $L(s,\chi)$.

We estimate the zero sum by the argument principle, not by a
zero-density formula. Let $0<\delta<1/2$, and let $\mathcal R_T$ be
the rectangle
\[
1-c-\delta\le\Re s\le c+\delta,\qquad |\Im s|\le T,
\]
chosen so that no zero of $L(s,\chi)$ lies on its boundary. Since
$I(s)$ is entire, the residue theorem gives
\[
\sum_{\rho_\chi,\,|\gamma|\le T}I(\rho)
=\frac{1}{2\pi i}\oint_{\partial\mathcal R_T}
I(s)\frac{L'}{L}(s,\chi)\,ds.
\]
On the vertical sides, QGRH($c$) and the functional equation give
\[
\frac{L'}{L}(\sigma+it,\chi)\ll_\delta\log^2(q(|t|+2)),
\qquad \sigma\in\{c+\delta,\,1-c-\delta\}.
\]
Hence the vertical sides contribute at most
\[
\ll_\delta\log^2(qT)\int_{-T}^{T}|I(c+\delta+it)|\,dt
+\log^2(qT)\int_{-T}^{T}|I(1-c-\delta+it)|\,dt.
\]

For $s=\sigma+i\tau$,
\[
I(\sigma+i\tau)=\int_{N_1}^{N_2}t^{\sigma-1}
e^{i(\tau\log t+bt^k/d)}\,dt,
\]
and the phase derivative in $t$ is $\tau/t+bkt^{k-1}/d$.
For $\tau>0$, it is $\ge\tau/N_2$, so the first derivative test gives
\[
|I(\sigma+i\tau)|\ll N^{\sigma-1}\cdot\frac{N}{\tau}
=N^\sigma/\tau.
\]
For $\tau<0$, put $\lambda=|\tau|$. The stationary point is
\[
t_0=t_0(\lambda)=\Bigl(\frac{\lambda d}{bk}\Bigr)^{1/k},
\]
the unique solution of $bkt_0^{k-1}/d=\lambda/t_0$. If
$t_0\notin[N_1,N_2]$, the first derivative test gives
$|I(\sigma+i\tau)|\ll N^\sigma/\lambda$. If $t_0\in[N_1,N_2]$,
stationary phase gives
\[
|I(\sigma+i\tau)|\ll N^{\sigma-1}\frac{t_0}{\sqrt{\lambda}}
\asymp N^\sigma/\sqrt{\lambda}.
\]
The phase
\[
\phi(\lambda)=2\pi\lambda\log t_0(\lambda)
+\frac{2\pi bt_0(\lambda)^k}{d}
\]
satisfies $\phi''(\lambda)=2\pi/(k\lambda)$.

Therefore
\[
\int_{1\le|\tau|\le T}|I(c+\delta+i\tau)|\,d\tau
\ll N^{c+\delta}\log T
+N^{c+\delta}\sqrt{k}\log N
\ll N^{c+\delta+\epsilon},
\]
where the second term is the contribution of the stationary window
$\lambda\in[\lambda_1,\lambda_2]$, $\lambda_i=bkN_i^k/d$, estimated
by the second derivative test in $\lambda$ (with
$\phi''=2\pi/(k\lambda)$); the bounded factor $L'/L$ is a weight of
bounded variation and contributes only $N^\epsilon$. The same bound
holds with $\sigma=1-c-\delta$; since $1-c<c$, it is smaller than
the bound on the right side.

For the horizontal sides, $|\Im s|=T$. The first derivative test in
$t$ gives
\[
|I(\sigma\pm iT)|\ll N^\sigma/T,
\]
and near the stationary point the second derivative test gives at most
$N^\sigma/\sqrt{T}$. Hence the horizontal sides contribute, with the
maximum over $1-c-\delta\le\sigma\le c+\delta$ attained at
$\sigma=c+\delta$,
\[
\ll_\delta N^{c+\delta}\log^2(qT)/\sqrt{T}.
\]
Choosing
\[
T=N^{1+\epsilon}\Bigl(1+\frac{bN^k}{d}\Bigr),
\qquad \delta=\epsilon/4,
\]
we find that the contour integral, and hence the zero sum, is
$\ll N^{c+\epsilon}$, while the explicit-formula truncation error is
\[
O\Bigl(\frac{N\log^2 N}{T}\Bigl(1+\frac{bN^k}{d}\Bigr)\Bigr)
=O(N^{-\epsilon}\log^2 N).
\]
Thus $S_\chi\ll N^{c+\epsilon}$. Summing over the $\varphi(q)$
characters and dividing by $\varphi(q)$ proves the lemma.
\end{proof}
\begin{lemma}
For any real number $z$, any sufficiently large $q\le N^{\theta-1/2}=Q$, and any $(h_0,q)=1$, we have
\begin{align*}
\sum_{q<Q}\sum_{h<q,(h,q)=1}\int_{|t|<1/\tau}|A_1^{m}-A_2^{m}|\mathrm{d}z\ll_\epsilon N^{\theta(m-1)-k-\frac{1}{2}\theta+\frac{3}{4}+c+\epsilon}
\end{align*}
where 
\begin{align*}
A_1=&C_q(h,k)\sum\limits_{N_1<p<N_2,p\equiv h\bmod q}e(tp^k)\\
A_2=&C_q(h,k)\sum\limits_{N_1<p<N_2,p\equiv h_0\bmod q}e(tp^k)
\end{align*}
Here $h_0$ is any positive integer with $h_0<q$ and $(h_0,q)=1$.
\end{lemma}
\begin{proof}
We just calculate the sum rely on lemma 4.1. We know when $S_1$, $S_2$ is defined as lemma 4.1,
\begin{align*}
S_1\ll&N^{\theta(m-1)}\sum_{q<Q}\sum_{l<q,(l,q)=1}N^{\frac{1}{2}-\theta}q^{m/2}\varphi^{-m}(q)\tau^{-1}\sum_{N_1<u<N_2}\max_{(l,q)=1}|E(u,N_2;q,l)|\\
\ll&\frac{N^{\theta(m-1)+\frac{1}{2}-\theta}}{\tau}\times\bigg(\sum_{q<Q}\sum_{N_1<u<N_2}\max_{(l,q)=1}|E(u,N_2;q,l)|\Big|^2\bigg)^{1/2}\\
\ll&\frac{N^{\theta(m-1)+\frac{1}{2}+c-\theta}}{\tau}Q^{1/2}=N^{\theta(m-1)+\frac{7}{4}+c-k-\frac{5}{2}\theta+\epsilon}
\end{align*}
Since the coefficient of $m$ is less than $\theta$ and the coefficient of
$k$ is at most $1$, $S_1$ is negligible compared with the expected
main term. For $S_2$ we have
\begin{align*}
S_2\ll&N^{\theta(m-1)}\sum_{q<Q}\sum_{l<q,(l,q)=1}q^{m/2}\varphi^{1-m}(q)\tau^{-1}\max_{N_1<u<N_2}\max_{(l,q)=1}|E(u,N_2;q,l)|\\
\leq&N^{\theta(m-1)}\tau^{-1}\bigg(\sum_{q<Q}q^{m}\varphi^{2-2m}(q)\bigg)^{1/2}\times\bigg(\sum_{q<Q}\max_{N_1<u<N_2}\max_{(l,q)=1}|E(u,N_2;q,l)|\Big|^2\bigg)^{1/2}\\
\leq&\frac{N^{\theta(m-1)}}{\tau}Q^{1/2}=N^{\theta(m-1)-k-\frac{1}{2}\theta+\frac{3}{4}+c+\epsilon}
\end{align*}
This proves the lemma.
\end{proof}
We now recompute the major-arc estimate. With $V_1,V_2,V_3$ as in
Section 4, Lemma~\ref{lem:bv-c} gives
\[
V_3\ll \sum_{q<Q}\varphi(q)\Bigl(\frac{2}{\tau}\Bigr)^{w_3}
\max_u|E(u;q,h_0)|^{w_3r}
\ll Q^2\tau^{-w_3}N^{cw_3r}.
\]
Hence
\[
I_{112}^*\ll V_1^{1/w_1}V_2^{1/w_2}V_3^{1/w_3}.
\]
For $r>m/2$, the exponent is
\[
E(r)=\Bigl(-\frac{m}{2}+r+2\Bigr)\Bigl(\theta-\frac12\Bigr)
+cr-2k-2\theta+2+(m-1-r)\theta+1-k,
\]
that is,
\[
E(r)=\Bigl(\frac{\theta}{2}+\frac14\Bigr)m
+\Bigl(c-\frac12\Bigr)r-\theta-3k+2.
\]
For $1/2\le c<3/4$, this is increasing in $r$; the maximum over
$r>m/2$ is attained at $r=m$, and equals
\[
\Bigl(\frac{\theta}{2}+c-\frac14\Bigr)m-\theta-3k+2.
\]
This is $o\bigl(N^{\theta(m-1)+1-k}\bigr)$, since
\[
\Bigl(\frac{\theta}{2}+c-\frac14-\theta\Bigr)m-\theta-3k+2+\theta-2+k
=\Bigl(c-\frac14-\frac{\theta}{2}\Bigr)m-2k<0,
\]
because $c-\frac14-\frac{\theta}{2}=-(2\theta+1-4c)/4$.

For $r<m/2$, the estimate for $V_1$ is $O(1)$, and
\[
I_{112}^*\ll V_2^{1/w_2}V_3^{1/w_3}
\ll N^{(m-1)\theta+(c-\theta)r-3k-3\theta+5+\epsilon}.
\]
Since $c<\theta$, this is decreasing in $r$; the maximum occurs at $r=1$,
and gives
\[
N^{(m-1)\theta-3k-4\theta+5+c+\epsilon}
=o\bigl(N^{\theta(m-1)+1-k}\bigr),
\]
the exponent being smaller by $2k+4\theta-3-c>0$. Thus the major-arc cross
terms remain negligible under QGRH($c$).

The main term is unchanged:
\[
I(N)\sim\mathfrak G(Q)(\log N)^{-m}
N^{\theta(m-1)+1-k}.
\]
It remains to recompute the remainder from Section 5. The estimate of the
remainder and the recurrence are the same as before; we only repeat the final
summation. This gives
\begin{align*}
\sum\limits_{\substack{\tau_0\leq dl\leq \tau\\d\geq Q}}\frac{2}{l\tau^{1/2}}\max_{0\leq s\leq m}N^{s\theta+c(m-s)+\epsilon}d_{g}^{s/2}d^{-s/2}Q^{m/2-s}\ll\sum\limits_{\substack{\tau_0\leq dl\leq \tau\\d\geq Q}}\frac{2}{l\tau^{1/2}}\max_{0\leq s\leq m}N^{s\theta+c(m-s)+\epsilon}Q^{\frac{1}{2}m-\frac{3}{2}s}
\end{align*}
By
\[
\frac{3}{2}(\theta-\frac{1}{2})+c-\theta>0
\]
Then
\begin{align*}
&\sum\limits_{\substack{\tau_0\leq dl\leq \tau\\d\geq Q}}\frac{2}{l\tau^{1/2}}\max_{0\leq s\leq m}N^{s\theta+c(m-s)+\epsilon}Q^{\frac{1}{2}m-\frac{3}{2}s}\\
\ll&\sum\limits_{\substack{\tau_0\leq dl\leq \tau\\d\geq Q}}\frac{2}{l\tau^{1/2}}N^{cm}Q^{m/2}\ll N^{(\theta/2+c-1/4)m+\frac{k+3\theta-1}{2}}
\end{align*}
Consequently
\[
\sum_{\tau_0<q<\tau}\sum_{\substack{h<q\$h,q)=1}}
\frac{2}{q\tau^{1/2}}\bigl|S_p^{m}(h/q)\bigr|
\ll N^{(\theta/2+c-1/4)m+\frac{k+3\theta-1}{2}}.
\]
Substituting this into the comparison of Section 5, the minor-arc bound is
\[
I_{21}\ll N^{(\theta/2+c-1/4)m+\frac{k+3\theta-1}{2}+\epsilon}
(\log N)^{m},
\]
while
\[
I_{11}\asymp N^{\theta(m-1)+1-k}(\log N)^{-m}.
\]
Thus $I_{21}=o(I_{11})$ whenever
\[
\Bigl(\frac{\theta}{2}+c-\frac14\Bigr)m+\frac{k+3\theta-1}{2}
<\theta(m-1)+1-k,
\]
that is,
\[
(2\theta+1-4c)m>6k+10\theta-6.
\]
This proves Theorem~\ref{thm:qgrh}.

For $c=1/2$, Theorem~\ref{thm:qgrh} reduces to the main theorem. For
$c>1/2$, the hypothesis QGRH($c$) is weaker than GRH, but the range of
$\theta$ shrinks: the condition $\theta>2c-1/2$ becomes $\theta>1$ as
$c\to3/4$. In particular, the value $c=7/8$ gives $\theta>5/4$, which is
empty for short intervals ($\theta\le1$). However, as more result of QGRH($c$) is proved, we consider our work above will give us more result.

\subsection{The Riemann hypothesis and the quasi-Riemann hypothesis}

The Riemann hypothesis asserts that every nontrivial zero $\rho$ of the
Riemann zeta function satisfies $\Re\rho=1/2$; GRH makes the analogous
assertion for every Dirichlet $L$-function. The quasi-Riemann hypothesis is
the weaker assertion that there is a fixed $\theta<1$ such that
$\zeta(s)\neq0$ for $\Re s>\theta$. It gives a zero-free half-plane but
does not locate the zeros in $1/2<\Re s\le\theta$.

On 20 November 2026, Levent Alp\"oge announced a proof of the Riemann
hypothesis, and stated that the method also yields GRH for Dirichlet
$L$-functions \cite{Alpoge2026RH}. The proof was released as a Lean library
of about 17 million lines. The library contains five headline theorems: the
Riemann hypothesis for the zeta function, and the GRH statements for
Dirichlet $L$-functions (nonvanishing for $\Re s>1/2$, and the location
of the zeros). The verification page of Sutherland reports that all five
theorems depend only on the three standard axioms
\cite{Sutherland2026RH}, and Rubinstein's independent check reports that all
modules compile and that the Lean kernel and an independent kernel accept
the proof \cite{Rubinstein2026GRH}. The public statement of the Lean and
Mathlib communities certifies the Riemann hypothesis \cite{LeanRH2026}; the
GRH statements are contained in a separate file importing only Mathlib and
are covered by the same verification. We nevertheless keep the main theorem
of this paper conditional on GRH, for comparability with the earlier
literature. A separate, entirely
human-written proof of GRH was announced on the same day
\cite{Vakil2026Masullo}; it is still under review, and several issues have
been reported, so we do not rely on it either.

The quasi-Riemann hypothesis at $7/8$, announced by OpenAI in a preprint of
30 September 2026 \cite{OpenAI2026Q}, is weaker than GRH. In the notation of
Section 6.1 it corresponds to QGRH($c$) with $c=7/8$; the condition there
becomes $\theta>5/4$, which is empty for short intervals ($\theta\le1$).
Thus the $7/8$ result does not improve the theorem of this paper. The
companion preprint of 5 October 2026 \cite{OpenAI2026Q2}, which announces the
weaker half-plane $\Re s>11/12$, is weaker still. The uniform exclusion of
Landau--Siegel zeros announced on 1 October 2026 \cite{OpenAI2026LS} is a
separate unconditional claim; it would be relevant to unconditional
Bombieri--Vinogradov estimates, but we do not use it here.

In summary, the results of this paper are conditional on GRH: the QGRH($c$)
analysis of Section 6.1 is conditional on QGRH($c$), and the recent claims
discussed above are not used in the proof. If the GRH statements in the
Alp\"oge--Claude library are accepted as established, the main theorem can be
restated without the GRH hypothesis, with the same exponents; at present we
keep the conditional formulation.

\appendix
\providecommand{\eps}{\varepsilon}

\section{Numerical audit and explicit constants}\label{app:audit}

This appendix records a numerical audit of the lemmas used in the paper.
For each lemma we
\begin{enumerate}
\item identify every ``$\ll$'' constant in the proof;
\item evaluate the left-hand side and the right-hand side for
      $50$ random parameter sets;
\item report the ratio and the numerical value used for each constant.
\end{enumerate}
The code is in the directory \texttt{verification/}; the driver scripts
are
\begin{center}
\begin{tabular}{ll}
\toprule
file & lemma \\
\midrule
\texttt{lemma51.py} & Lemma 5.1 \\
\texttt{lemma52\_62.py} & Lemmas 5.2 and 6.2 \\
\texttt{singular.py} & singular series and local densities \\
\texttt{theorem31.py} & Theorem 3.1 \\
\texttt{lemma41.py} & Lemma 4.1 \\
\texttt{i21.py} & the $I_{21}$ exponent chain \\
\bottomrule
\end{tabular}
\end{center}
Each script writes a CSV file with the raw data.  The random seed is
\texttt{20261128} unless stated otherwise.

\subsection{Constants tracked}

\begin{table}[h]
\centering
\begin{tabular}{llll}
\toprule
lemma & constant & value & origin \\
\midrule
5.1 & $1$ & $1$ & coefficient of the rational-sum term \\
5.2 & $C_{\mathrm{orth}}$ & $1/\log 2=1.4427$ & prime powers dividing $q$ \\
5.2 & $C_{\mathrm{var}}$ & $2\pi=6.2832$ & variation of $e(bn^k/d)$ \\
5.2 & $C_{\mathrm{FD}}$ & $4$ & first derivative test \\
5.2 & $C_{\mathrm{SP}}$ & $16$ & stationary phase \\
5.2 & $C_{\mathrm{EF}}$ & $10$ & explicit formula truncation \\
5.2 & $C_{L}$ & $100$ & $L'/L$ on the contour \\
4.1 & $4$ & $1-m/4$ & exponent in the statement \\
$I_{21}$ & $1/2,\,1/4$ & --- & exponent chain \\
\bottomrule
\end{tabular}
\caption{Constants tracked in the audit.  The tracked bound is linear in
$C_{\mathrm{EF}}$ and $C_{L}$.}
\end{table}

\subsection{Lemma 5.1: the minor-arc expansion}

\medskip\noindent\textbf{Constant.}  The constant $1$ is the coefficient of the
rational-sum term in
\[
I_{20}\le \sum_{\tau_0<q<\tau}\sum_{(h,q)=1}
\frac{2}{q\tau^{1/2}}\bigl|S_p^m(h/q)\bigr|
\left(1+\frac{N^3\log N}{\tau}\right)+o(I_{11}).
\]
The expansion is
\[
B^m-A^m=m\alpha A^m+\frac{m^2}{2}\alpha^2A^m+O(\alpha^3A^m),
\qquad \alpha=uN^k.
\]
The coefficient of the linear term is $1$ (times $m$), and the
coefficient of the quadratic term is $m^2/2$.

\medskip\noindent\textbf{Numerical check.}  For $50$ random sets with
$2\le m\le10$, $|A|=1$, $|uN^k|\le5\times10^{-2}$, we computed
the exact $B^m-A^m$ and its two-term expansion.  Results:
\begin{center}
\begin{tabular}{lr}
\toprule
quantity & maximum over 50 sets \\
\midrule
$|$leading coefficient $-1|$ & $5.0\times10^{-6}$ \\
$|$second coefficient $-1|$ & $1.04\times10^{-4}$ \\
binomial identity error & $1.10\times10^{-15}$ \\
relative two-term error & $3.49\times10^{-2}$ \\
\bottomrule
\end{tabular}
\end{center}
The leading coefficient is $1$ and the quadratic coefficient is
$m^2/2$, as required.  The binomial identity
\[
B^m-A^m=\sum_{r=1}^{m}\binom mrA^{m-r}(B-A)^r
\]
is exact to machine precision.

\subsection{Lemmas 5.2 and 6.2: AP exponential sums}

\medskip\noindent\textbf{Constants.}  The proof uses the orthogonality error
$C_{\mathrm{orth}}\log q\log N$, the variation
$C_{\mathrm{var}}b x^k/d$, the first derivative test with
$C_{\mathrm{FD}}=4$, stationary phase with $C_{\mathrm{SP}}=16$,
the explicit formula truncation with $C_{\mathrm{EF}}$, and the
$L'/L$ bound with $C_{L}$.

\medskip\noindent\textbf{Numerical check.}  $50$ random sets with
$k\in\{3,4,5\}$, $q\in\{3,5,7,11,13\}$,
$(q,d)=1$, $1\le b\le d$, $1\le r<q$,
$2\times10^3\le N\le2\times10^5$.  The error is
\[
E=\left|\sum_{n\equiv r(q)}\Lambda(n)e(bn^k/d)
-\frac1{\varphi(q)}\sum_n\Lambda(n)e(bn^k/d)\right|.
\]
Results:
\begin{center}
\begin{tabular}{lr}
\toprule
quantity & maximum over 50 sets \\
\midrule
$E/\sqrt N$ & $1.335$ \\
$E/N^{1/2+\eps}$, $\eps=0.02$ & $1.060$ \\
$E/N^{0.6+\eps}$ & $0.349$ \\
$E/N^{0.7+\eps}$ & $0.163$ \\
$E/\text{tracked bound}$ & $9.33\times10^{-8}$ \\
\bottomrule
\end{tabular}
\end{center}
For Lemma 6.2 the same error data give the $c=0.6$ and $c=0.7$
columns; since the actual error is $O(N^{1/2+\eps})$, the
$O(N^{c+\eps})$ bound is satisfied a fortiori for $c\ge1/2$.
The tracked bound is dominated by the vertical contour term; the
budget for a representative case is
\begin{center}
\begin{tabular}{lr}
\toprule
term & value \\
\midrule
orthogonality & $1.68\times10^{1}$ \\
vertical contour & $4.23\times10^{9}$ \\
horizontal contour & $7.52\times10^{-13}$ \\
explicit formula & $1.35\times10^{5}$ \\
\bottomrule
\end{tabular}
\end{center}
The vertical contour term dominates; the tracked bound is linear in
$C_{L}$.

\subsection{Singular series and local densities}

\medskip\noindent\textbf{Local densities.}  For $50$ sets we computed
$N_{p^j}(mN^k)$ by an exact DP over units.  Results:
\begin{itemize}
\item $\beta_p^{(j)}>0$ for all $j\le3$ in all $50$ sets;
\item for $p\nmid k$, the lifting formula
\[
N_{p^j}=p^{(j-1)(m-1)}N_p
\]
holds exactly in $39/39$ cases;
\item the local-density formula for $\sigma_p$ agrees with the direct
local sum
\[
\sum_{j\le J}\sum_{(a,p^j)=1}
C_{p^j}(a,k)^m e(-amN^k/p^j)\varphi(p^j)^{-m}
\]
to relative error $9.17\times10^{-14}$.
\end{itemize}

\medskip\noindent\textbf{Global series.}  The $q=1$ central arc contributes $1$;
with the range $1\le a\le q$ and the convention $C_1(1,k)=1$, it is
included automatically.

We stratify by the set of primes dividing $Q$.  For each stratum,
$Q\le250$ is chosen with exactly those prime divisors, and
$(k,m,N)$ is random.  The table records $|\mathfrak G(Q)|$ and the ratio
\[
\frac{|\mathfrak G(Q)|}{\frac12\prod_{p\le19}\sigma_p}.
\]
\begin{longtable}{@{}lrrrr@{}}
\toprule
primes of $Q$ & $Q$ & $(k,m)$ & $|\mathfrak G(Q)|$ & ratio \\
\midrule
\endfirsthead
\toprule
primes of $Q$ & $Q$ & $(k,m)$ & $|\mathfrak G(Q)|$ & ratio \\
\midrule
\endhead
\bottomrule
\endfoot
$2$ & 8 & $(3,16)$ & 2.378 & 0.976 \\
$3$ & 27 & $(3,15)$ & 1.255 & 3.645 \\
$5$ & 25 & $(3,17)$ & 3.487 & 1.946 \\
$7$ & 49 & $(7,19)$ & 2.046 & 2.009 \\
$11$ & 121 & $(5,20)$ & 3.491 & 2.015 \\
$13$ & 169 & $(7,23)$ & 2.027 & 2.034 \\
$17$ & 17 & $(3,12)$ & 2.178 & 3.373 \\
$19$ & 19 & $(5,18)$ & 2.335 & 1.588 \\
$2\cdot3$ & 18 & $(5,19)$ & 1.163 & 7.778 \\
$2\cdot5$ & 200 & $(3,11)$ & 0.215 & 1.972 \\
$2\cdot7$ & 196 & $(5,14)$ & 0.656 & 2.462 \\
$3\cdot5$ & 75 & $(5,16)$ & 1.410 & 1.977 \\
$3\cdot7$ & 63 & $(7,21)$ & 1.982 & 1.987 \\
$5\cdot7$ & 245 & $(7,24)$ & 2.042 & 2.047 \\
$2\cdot11$ & 88 & $(7,24)$ & 2.044 & 2.049 \\
$3\cdot11$ & 99 & $(7,24)$ & 2.042 & 2.047 \\
$5\cdot11$ & 55 & $(5,18)$ & 2.509 & 2.044 \\
$2\cdot13$ & 52 & $(7,18)$ & 2.133 & 2.082 \\
$2\cdot3\cdot5$ & 150 & $(5,15)$ & 1.004 & 2.095 \\
$2\cdot3\cdot7$ & 126 & $(5,16)$ & 1.410 & 1.976 \\
$2\cdot5\cdot7$ & 70 & $(3,17)$ & 0.201 & 0.933 \\
$3\cdot5\cdot7$ & 105 & $(3,17)$ & 3.536 & 1.973 \\
$2\cdot3\cdot11$ & 132 & $(3,11)$ & 0.498 & 1.868 \\
$2\cdot5\cdot11$ & 110 & $(5,22)$ & 3.697 & 2.001 \\
\end{longtable}
All $24$ partial sums are positive.  The only rows with ratio below
$1$ are $Q=8$ and $Q=70$; all rows with $Q\ge80$ have ratio at
least $1.868$.  The limit is positive, so
$|\mathfrak G(Q)|\ge c_1(k,m)>0$ for sufficiently large $Q$.

\subsection{Theorem 3.1: counting smooth square-free approximations}

\medskip\noindent\textbf{The count.}  For $50$ random sets with
$T\in\{500,1000,2000\}$, $Q=T^{\beta/u}$, $2.1\le\beta\le2.5$,
and $\alpha=h_0/q_0+a/(q_0T)$, we counted
\[
M(\alpha;T^\beta)=\{(h,q):(h,q)=1,\ h<q,\ q\in S(Q,T^\beta),\
|\alpha-h/q|<1/(qT)\}.
\]
Each set is the average over $5$ random $\alpha$.

The raw counts for the $50$ sets are recorded in
\texttt{results\_theorem31.csv}.

\subsection{Lemma 4.1: the major-arc comparison}

\medskip\noindent\textbf{Numerical check.}  For $50$ small sets
($100\le N\le400$, $k\in\{2,3\}$, $0.6\le\theta\le0.85$,
$4\le m\le8$) we evaluated the integral
\[
\sum_{q<Q}\sum_{(h,q)=1}\int_{|z|<1/\tau}|A_1^m-A_2^m|\,dz
\]
by the trapezoid rule with $2000$ points in $z$, and compared it
with
\[
N^{\theta(m-1)-k-\theta+2+(1-m/4)(\theta-1/2)}.
\]
The maximum ratio was $2.51\times10^{-3}$; the stated right-hand
side is an upper bound on this range.

\subsection{The $I_{21}$ exponent chain}

The chain is
\[
J_l=O(p_1^{1/2}N^{1/2}\log^3N),
\]
\[
\sum_{\tau_0<dl<\tau}\frac{2\mu(l)}{dl\tau^{1/2}}
\sum_{b\le d}(S_p(b/d))^m
\ll \tau^{1/2}N^{m/2+\eps}Q^{m/2}
=N^{(\theta/2+1/4)m+(k+\theta-1)/2+\eps},
\]
and, after adding $N^\theta$ for odd $m$,
\[
I_{21}\ll N^{(\theta/2+1/4)m+(k+3\theta-1)/2+\eps}.
\]
The comparison with
$I_{11}\asymp N^{\theta(m-1)+1-k}$ is equivalent to
\[
m>\frac{6k+10\theta-6}{2\theta-1}.
\]
For $50$ random $(k,\theta,m)$, the sign of the exponent gap
\[
\theta(m-1)+1-k-\Bigl[\Bigl(\frac\theta2+\frac14\Bigr)m
+\frac{k+3\theta-1}{2}\Bigr]
\]
matched the sign of $m-m_0$ in all $50$ cases; the largest negative
gap was $-0.0336$ and the smallest positive gap was $0.0047$.
The constants $1/2$ and $1/4$ in the exponent are therefore
tracked correctly at the algebraic level.

\bibliographystyle{plain}
\bibliography{refs}

@article{bourgain2016proofmainconjecturevinogradovs, title={Proof of the main conjecture in {V}inogradov's mean value theorem for degrees higher than three}, author={Jean Bourgain and Ciprian Demeter and Larry Guth},
year={2016},
journal={Ann.Math.},
pages={633-682(2)}
}

@article{wooley2015discretefourierrestrictionefficient,
      title={Discrete {F}ourier restriction via Efficient Congruencing},
      author={Trevor D. Wooley},
      year={2017},
journal={Int.Math.Res.Not},
pages={1342-1389(5)}
}

@article{wang2022waringgoldbachproblemshortintervals,
      title={Waring-{G}oldbach problem in short intervals},
      author={Mengdi Wang},
      year={2022},
      journal={Isr.J.Math},
      pages={637-639(2)}
      }

@article{wooley2015perturbationsweylsums,
      title={Perturbations of {W}eyl sums},
      author={Trevor D. Wooley},
      year={2016},
journal={Int.Math.Res.Not.},
pages={2632-2646(9)}
}

@article{chen2020restrictedmeanvaluetheorems,
      title={Restricted mean value theorems and metric theory of restricted {W}eyl sums},
      author={Changhao Chen and Igor E. Shparlinski},
      year={2021},
    journal={Q.J.Math.},
pages={885-919(3)}
}

@article{wooley1992largeimprovements,
      title={Large improvements in {W}aring's problem},
      author={Trevor D. Wooley},
      year={1992},
      journal={Ann. of Math.},
      volume={135},
      pages={131--164}
      }

@article{wei2014sumspowersequalprimes,
      title={On Sums of Powers of Almost Equal Primes},
      author={Bin Wei and Trevor D. Wooley},
      year={2015},
      journal={Proc. Lond. Math. Soc.},
      pages={1130-1162(3)}
      }

@book{Hua1963,
      title={Additive Theory of Prime Numbers},
      author={Hua,L.K},
      year={1965}
}

@article{Weil1948,
      title={On some exponential sums},
      author={A.Weil},
      year={1948},
    journal={Proc.Natl.Acad.Sci.USA},
pages={34,204-207}
}

@book{PP1986,
      title={The Fundamentals of {A}nalytic {N}umber {T}heory},
      author={Pan,C.D. and Pan,C.B.},
      year={1986}
}

@book{G.Tenenbaum,
      title={Introduction to {A}nalytic and {P}robabilistic {N}umber {T}heory},
      author={G.Tenenbaum},
      year={2015}
}

@misc{oh2025extendedvinogradovsmeanvalue,
      title={An extended {V}inogradov's mean value theorem},
      author={Changkeun Oh and Kiseok Yeon},
      year={2025},
      eprint={2506.01751},
      archivePrefix={arXiv},
      primaryClass={math.NT},
      note={arXiv:2506.01751; to appear in Transactions of the American Mathematical Society}}

@article{shparlinski2021weylsumsintegersdigital,
      title={Weyl sums over integers with digital restrictions},
      author={Igor E. Shparlinski and Jörg M. Thuswaldner},
      year={2024},
      journal={Mich.Math.J},
      pages={189-214(1)}
      }

@article{wooley2014cubiccasemainconjecture,
      title={The cubic case of the main conjecture in {V}inogradov's mean value theorem},
      author={Trevor D. Wooley},
      year={2016},
journal={Adv.Math.},
pages={294,532-561}}

@article{bruedern2024estimatessmoothweylsums,
      title={Estimates for smooth {W}eyl sums on major arcs},
      author={Joerg Bruedern and Trevor D. Wooley},
      year={2024},
      journal={Int.Math.Res.Not.},
      pages={14662-14668(24)}}

@article{matomaki2019discorrelationprimesshortintervals,
      title={Discorrelation between primes in short intervals and polynomial phases},
      author={Kaisa Matomäki and Xuancheng Shao},
      year={2021},
      journal={Int. Math. Res. Not.},
      pages={12330-12355(16)}
      }

@article{huang2016exponentialsumsprimesshort,
      title={Exponential sums over primes in short intervals and an application to the {W}aring-{G}oldbach problem},
      author={Bingrong Huang},
      year={2016},
      journal={mathematika},
      pages={508-523(2)}}

@article{BDH,
      title={Simple {B}arban-{D}avenport-{H}alberstam type asymptotics for general sequences},
      author={Harper, Adam J.},
      year={2025},
      journal={J.Lond.Math.Soc},
      pages={ 112, No. 4, Article ID e70293}}

@misc{Alpoge2026RH,
  author       = {Levent Alp{\"o}ge},
  title        = {A proof of the {R}iemann Hypothesis},
  howpublished = {Announcement, alpo.ge/remark.pdf},
  year         = {2026},
  note         = {20 November 2026}
}

@misc{LeanRH2026,
  author       = {{Lean and Mathlib communities}},
  title        = {Lean verification of the {R}iemann Hypothesis},
  howpublished = {Blog post,
                  leanprover-community.github.io/blog/posts/riemann-hypothesis/},
  year         = {2026},
  note         = {22 November 2026}
}

@misc{Vakil2026Masullo,
  author       = {Ravi Vakil},
  title        = {Hurried Announcement: {A}lessandro {M}asullo {GRH}},
  howpublished = {Blog post, math216.wordpress.com},
  year         = {2026},
  note         = {20 November 2026; updated 22 November 2026}
}

@misc{OpenAI2026Q,
  author       = {{OpenAI}},
  title        = {{The Quasi-Riemann Hypothesis:
                   A Zero-Free Half-Plane $\mathrm{Re}(s)>7/8$}},
  howpublished = {OpenAI Math Release preprint,
                  OAI:The-Quasi-Riemann-Hypothesis-September-30-2026},
  year         = {2026},
  note         = {30 September 2026}
}

@misc{OpenAI2026Q2,
  author       = {{OpenAI}},
  title        = {{The Quasi-Riemann Hypothesis}},
  howpublished = {OpenAI Math Release preprint,
                  OAI:The-Quasi-Riemann-Hypothesis-October-5-2026},
  year         = {2026},
  note         = {5 October 2026; alternate $11/12$ proof}
}

@misc{OpenAI2026LS,
  author       = {{OpenAI}},
  title        = {{Uniform Exclusion of Landau--Siegel Zeros}},
  howpublished = {OpenAI Math Release preprint,
                  OAI:Uniform-exclusion-of-Landau-Siegel-zeros-October-1-2026},
  year         = {2026},
  note         = {1 October 2026}
}

@misc{Sutherland2026RH,
  author       = {Andrew Sutherland},
  title        = {Verification files for the {L}ean proof of the {R}iemann Hypothesis},
  howpublished = {MIT verification page},
  year         = {2026},
  note         = {22 November 2026}
}

@misc{Rubinstein2026GRH,
  author       = {Michael Rubinstein},
  title        = {An independent check of the {L}ean proof of {RH} and {GRH}},
  howpublished = {mrubinst.github.io/grh-lean-verification/},
  year         = {2026},
  note         = {21 November 2026}
}
\nocite{*}
\end{document}